\documentclass[11pt]{amsart}
\usepackage[margin=1in,top=1in,centering]{geometry}
\usepackage[utf8]{inputenc}
\usepackage[english]{babel}
\usepackage{amsmath,amssymb,amsthm}
\usepackage{tikz}
\usepackage{comment}
\usepackage{mathtools}
\usepackage{adjustbox}
\usepackage{graphicx}
\usetikzlibrary{cd}

\usepackage[
colorlinks=true,
urlcolor=purple,
linkcolor=purple!87!black,
pdfborder={0 0 0}
]{hyperref}

\newcommand{\Z}{\mathbb{Z}}

\newcommand{\N}{\mathbb{N}}

\newcommand{\D}{\mathcal{D}}
\newcommand{\F}{\mathcal{F}}

\newcommand{\add}{\mathrm{add}\,}
\newcommand{\Ext}{\mathrm{Ext}}
\newcommand{\Hom}{\mathrm{Hom}}
\newcommand{\dell}{\mathrm{dell}\,}
\newcommand{\kdell}{k\text{-}\mathrm{dell}\,}
\newcommand{\ddell}{\mathrm{ddell}\,}
\newcommand{\edell}{\mathrm{edell}\,}
\newcommand{\kedell}{k\text{-}\mathrm{edell}\,}
\newcommand{\kddell}{k\text{-}\mathrm{ddell}\,}
\newcommand{\subddell}{\mathrm{sub}\text{-}\mathrm{ddell}\,}
\newcommand{\ksubddell}{k\text{-}\mathrm{sub}\text{-}\mathrm{ddell}\,}
\newcommand{\dellT}{\mathrm{dell}_T}
\newcommand{\phidim}{\phi\dim}
\newcommand{\phidimT}{\phi_T\dim}
\newcommand{\findim}{\mathrm{findim}\,}
\newcommand{\Findim}{\mathrm{Findim}\,}

\newcommand{\op}{\mathrm{op}}
\newcommand{\pd}{\mathrm{pd}\,}
\newcommand{\id}{\mathrm{id}\,}

\newcommand{\dirsum}{\xhookrightarrow{\!\oplus}}

\newcommand{\soc}{\mathrm{soc}\,}
\newcommand{\smod}{\underline{\mathrm{mod}}\,}
\renewcommand{\mod}{\mathrm{mod}\,}
\newcommand{\Mod}{\mathrm{Mod}\,}
\newcommand{\coker}{\mathrm{coker}\,}
\newcommand{\im}{\mathrm{im}\,}
\newcommand{\inc}{\xhookrightarrow{}}
\newcommand{\onto}{\twoheadrightarrow}

\theoremstyle{plain}
\newtheorem{theorem}{Theorem}[section]
\newtheorem{definition}[theorem]{Definition}
\newtheorem{lemma}[theorem]{Lemma}
\newtheorem{proposition}[theorem]{Proposition}
\newtheorem{corollary}[theorem]{Corollary}
\newtheorem{remark}[theorem]{Remark}
\newtheorem{example}[theorem]{Example}
\newtheorem{question}[theorem]{Question}

\title{Derived Delooping Levels and Finitistic Dimension}

\author{Ruoyu Guo and Kiyoshi Igusa}


\begin{document}

\begin{abstract}
In this paper, we develop new ideas regarding the finitistic dimension conjecture, or the findim conjecture for short. Specifically, we improve upon the delooping level by introducing three new invariants called the effective delooping level $\edell\!$, the sub-derived delooping level $\subddell\!$, and the derived delooping level $\ddell\!$. They are all better upper bounds for the opposite Findim. Precisely, we prove
\[
\Findim\Lambda^{\op} = \edell\Lambda \leq \ddell\Lambda \text{ (or $\subddell\Lambda$)} \leq \dell\Lambda
\]
and provide examples where the last inequality is strict (including the recent example from \cite{kershaw2023} where $\dell\Lambda=\infty$, but $\ddell\Lambda=1=\Findim\Lambda^{\op}$).

We further enhance the connection between the findim conjecture and tilting theory by showing finitely generated modules with finite derived delooping level form a torsion-free class $\mathcal{F}$. Therefore, studying the corresponding torsion pair $(\mathcal{T}, \mathcal{F})$ will shed more light on the little finitistic dimension. Lastly, we relate the delooping level to the $\phi$-dimension $\phidim$, a popular upper bound for findim, and give another sufficient condition for the findim conjecture.
\end{abstract}

\maketitle

\tableofcontents

\section{Introduction and Main Results}

The finitistic dimension conjectures have been an active area of research for decades since they were first proposed by Rosenberg and Zelinsky in \cite{bass1960}. The original conjectures stated that, for any ring $\Lambda$,
\begin{enumerate}
\item $\Findim\Lambda = \findim\Lambda$, and
\item $\findim\Lambda<\infty$,
\end{enumerate}
where
$$\Findim\Lambda = \sup \{\pd M \mid M\in\Mod\Lambda, \, \pd M<\infty \},$$
$$\findim\Lambda = \sup \{\pd M \mid M\in\mod\Lambda, \, \pd M<\infty \},$$
$$\pd M = \inf \{n\in \N \mid \Omega^n M \text{ is projective} \}.$$

The first part of the conjectures is shown to fail in many situations, including even monomial algebras \cite{huisgen1992homological}. {\color{black} On the other hand, there is no known counterexample for the second part of the conjecture for finite dimensional $\Lambda$. In the infinite-dimensional case, however, if $\Lambda$ is a commutative Noetherian ring that is regular of infinite Krull dimension, then there exists an algebra $\tilde{\Lambda}$ related to $\Lambda$ such that $\findim\tilde{\Lambda}=\infty$ and $\findim\tilde{\Lambda}^{\op}=0$ \cite{krause2022symmetry}.} 

We will call the second part of the finitistic dimension conjectures restricted to finite dimensional algebras the \textbf{findim conjecture} for short. The finiteness of findim bears much homological significance. Some consequences of the findim conjecture are
\begin{enumerate}
\item \textbf{Wakamatsu tilting conjecture} \cite{mantese2004wakamatsu}

\begin{enumerate}
\item \textbf{Gorenstein symmetry conjecture}, a consequence of the Wakamatsu tilting conjecture
\end{enumerate}

\item \textbf{Strong Nakayama conjecture}, which is equivalent to the Nunke condition, by \cite{happel1993reduction}

\begin{enumerate}
\item \textbf{Generalized Nakayama conjecture} and \textbf{Auslander-Reiten conjecture} are equivalent \cite{yamagata1996frobenius} and are both implied by the strong Nakayama conjecture
\item \textbf{Nakayama conjecture}, implied by the generalized Nakayama conjecture \cite{auslander1975generalized}
\end{enumerate}
\end{enumerate}

Many methods were developed in an attempt to solve the findim conjecture, a lot of which involve the analysis of syzygies. This turns out to be the center of our attention in this paper as well. Another method to solve the findim conjecture involves injective generation \cite{rickard2019unbounded}. For a survey of some existing results and more implications among these conjectures, see \cite{masters_thesis, happel1991homological, huisgen1995}.

While the findim conjecture holds for many important algebras, such as monomial algebras \cite{GKK1991}, radical cube zero algebras \cite{green1991, igusa2005}, representation dimension three algebras \cite{igusa2005}, and special biserial algebras (since its representation dimension is 3 shown in \cite{erdmann2004radical}) and so on, it remains a very difficult question in general. New ideas arose when G\'elinas introduced the delooping level $\dell\Lambda$ of an algebra $\Lambda$ and showed $\Findim\Lambda^{\op}\leq \dell\Lambda$ \cite{gelinas2022}. Therefore, one way to solve the findim conjecture amounts to showing the delooping level is always finite. So naturally, we have a few questions.
\begin{enumerate}
\item Is the delooping level finite for all $\Lambda$?
\item Is the difference $\dell\Lambda-\Findim\Lambda^{\op}$ or $\dell\Lambda-\findim\Lambda^{\op}$ finite, zero, or arbitrarily large? 
\end{enumerate}

For Question 1, the answer is negative due to a recent counterexample by Kershaw and Rickard \cite{kershaw2023}, but we will show their example has finite derived delooping level, our improved invariant, in Section \ref{sec:trivial_extensions}. For Question 2, it is known that $\Findim\Lambda-\findim\Lambda$ can be arbitrarily large \cite{smalo1998}, so the same is true for $\dell\Lambda-\findim\Lambda^{\op}$. On the other hand, G\'elinas asked whether $\dell\Lambda=\Findim\Lambda^{\op}$ is always true in Question 4.2 of \cite{gelinas2022}, and we show that this is false with a motivating counterexample (Example \ref{ex:monomial_example}) of a monomial algebra in Section \ref{sec:main_theorems}, where $\dell\Lambda = \Findim\Lambda^{\op}+1$. The difference $\dell\Lambda-\Findim\Lambda^{\op}$ can indeed be arbitrarily large due to the same example in \cite{kershaw2023}.

In order to close the gap between $\dell\Lambda$ and $\Findim\Lambda^{\op}$, we introduce three new invariants called the \textbf{effective delooping level} $\edell\Lambda$, the \textbf{sub-derived delooping level} $\subddell\Lambda$, and the \textbf{derived delooping level} $\ddell\Lambda$ as improvements to the original delooping level. The improvement lies in the fact that computing $\ddell$ and $\subddell$ involves considering all modules; on the other hand, the delooping level only considers syzygies of simple modules $S$ and modules of the form $\Omega^{n+1}\raisebox{\depth}{\scalebox{1}[-1]{\(\Omega\)}}^{n+1}\Omega^n S$ \cite[Corollary 1.12]{gelinas2022}, where $(\raisebox{\depth}{\scalebox{1}[-1]{\(\Omega\)}}, \Omega)$ is an adjoint pair on $\smod\Lambda$. Our first theorem is as follows.

\begin{theorem}
\label{thm:theorem1}
For any finite dimensional algebra $\Lambda$ over a field $\mathbb{K}$,
\begin{equation}
\label{eq:thm1}
\Findim\Lambda^{\op} = \edell\Lambda \leq \ddell\Lambda \text{ (or $\subddell\Lambda$)} \leq \dell\Lambda.
\end{equation}
\end{theorem}

The upper bound given by the sub-derived delooping level $\subddell\Lambda$ considers all modules that every simple module $S$ can embed into. Specifically, we define for every $\Lambda$-module $M$
\[
\subddell M = \inf\{\dell N \mid M \inc N\},
\]
and prove a new upper bound $\subddell\Lambda$ for $\Findim\Lambda^{\op}$
\[
\subddell \Lambda = \sup \{ \subddell S \mid \text{$S$ is a simple $\Lambda$-module} \}.
\]
On the other hand, the upper bound given by the derived delooping level $\ddell\Lambda$ considers all exact sequences that end in a simple module $S$. We define for every $\Lambda$-module $M$ the more general $k$-delooping level
\[
\kdell M = \inf \{n\in\N \mid \Omega^n M \text{ is a direct summand of } \Omega^{n+k}N \text{ for some $N\in\smod\Lambda$} \},
\]
the derived delooping level
\begin{align*}
\ddell M = \inf \{m\in\N \mid & \,\exists n\leq m \text{ and an exact sequence in $\mod\Lambda$ of the form} \\
& \,\, 0 \to C_n \to C_{n-1} \to \cdots \to C_1 \to C_0 \to M \to 0, \\
& \text{ where $(i+1)$-$\dell C_i\leq m-i$, } i=0,1,\dots,n  \},
\end{align*}
and prove another new upper bound $\ddell\Lambda$ for $\Findim\Lambda^{\op}$
\[
\ddell\Lambda = \sup \{ \ddell S \mid \text{$S$ is a simple $\Lambda$-module} \}.
\]

Moreover, the set $\mathcal{F}_1(\Lambda)$ of finitely generated $\Lambda$-modules with finite derived delooping level is closed under extensions, submodules, and direct sums (Theorem \ref{thm: finite ddell is torsion-free class}), making it a \textbf{torsion-free class} in $\mod\Lambda$. However, the set of finitely generated $\Lambda$-modules with finite delooping level is not closed under extensions by \cite{kershaw2023}. See Remark \ref{rem: dell not closed under ext} for details.

We also remark that the inequalities $\ddell\Lambda\leq\dell\Lambda$ and $\subddell\Lambda\leq\dell\Lambda$ can be strict as we will show in Examples \ref{ex:monomial example revisited subddell} and \ref{ex:monomial example revisited ddell}. More general definitions and theorems involving the derived delooping level and sub-derived delooping level are stated in Definitions \ref{def:subddell}, \ref{def:ddell} and Theorems \ref{thm:subddell} and \ref{thm:ddell}.

\begin{comment}
Moreover, all inequalities in \eqref{eq:thm1} exist in more generality. Specifically, we have for $k\in\Z_{>0}$
\begin{equation}
\label{eq:k_version_of_thm1}
\kedell\Lambda \leq \kddell\Lambda \text{ (or $\ksubddell\Lambda$)} \leq \kdell\Lambda,
\end{equation}
where the main difference from \eqref{eq:thm1} is that $\kedell\Lambda$ is not necessarily an upper bound for $\Findim\Lambda^{\op}$ for $k>1$. The other three terms in \eqref{eq:k_version_of_thm1}, however, all upper bound $\Findim\Lambda^{\op}$ as we will show that they increase as $k$ increases. The four inequalities of $\eqref{eq:k_version_of_thm1}$ are proved in Theorems \ref{thm:subddell}, \ref{thm:ddell} and Remarks \ref{rmk:subddell_vs_kdell}, \ref{rmk:ddell_less_than_dell}.

The significance of Theorem \ref{thm:theorem1} is the new interpretation of $\Findim\Lambda^{\op}$ as the effective delooping level of $\Lambda$. It also provides the derived and sub-derived delooping levels as better upper bounds than the delooping level.
\end{comment}

It is curious to see that the delooping level bounds the Findim of the \textit{opposite} algebra, while many other related invariants like the repetition index \cite{goodearl2014repetitive}, the $\phi$-dimension $\phidim$, and the $\psi$-dimension $\psi$-$\dim$ \cite{igusa2005}, all upper bound the findim of the \textit{same} algebra. Interestingly, we are able to compare the delooping level and the $\phi$-dimension of the \textit{same} algebra $\Lambda$, even if the finiteness of the former implies $\Findim\Lambda^{\op}<\infty$ and the latter $\findim\Lambda<\infty$. In fact, it seems more difficult to compare $\dell\Lambda$ and $\phidim\Lambda^{\op}$. Our second result provides a sufficient condition for the findim conjecture and compares $\dell\Lambda$ and $\phidim\Lambda$ in that case. The sufficient condition ($T_{\Lambda}$ is a finite set in Theorem \ref{thm:theorem2}) for the findim conjecture was already observed in \cite[Observation 2.5]{goodearl2014repetitive}, but we add the comparison between $\dell\Lambda$ and $\phidim\Lambda$.

\begin{theorem}
\label{thm:theorem2}
For any finite dimensional algebra $\Lambda$ over a field $K$, let $T_{\Lambda}$ be the set of non-projective indecomposable summands of syzygies of simple $\Lambda$-modules, including all simple modules. If $T_{\Lambda}$ is a finite set, then
\[
\Findim\Lambda^{\op} \leq \subddell\Lambda \text{ or } \ddell\Lambda \leq \dell \Lambda \leq \phidimT\Lambda \leq \phidim\Lambda,
\]
and in particular, since $\phidimT\Lambda$ is finite, the finitistic dimension conjecture holds for $\Lambda^{\op}$.
\end{theorem}

Here we present some background definitions for the rest of the paper. Throughout the paper, let $\Lambda$ be a finite dimensional algebra over a field $K$. Let $\mod\Lambda$ and $\Mod\Lambda$ be the category of finitely generated right $\Lambda$-modules and the category of all right $\Lambda$-modules, respectively. Let $\smod\Lambda$ be the \textbf{stable module category. Since all of our calculations are mod projectives, unless stated otherwise, exact sequences in the sequel may omit projective direct summands.} Given $\Lambda$, we use the term ``module'' to mean \textbf{right} $\Lambda$-modules, so left $\Lambda$-modules are $\Lambda^{\op}$-modules. Morphisms between modules are referred to as maps for simplicity. Since every module $M$ has its projective cover $P_M$ and injective envelope $I_M$, we define the \textbf{syzygy} $\Omega M$ of $M$ (resp. \textbf{cosyzygy} $\Sigma M$ of $M$), as the kernel of the surjection $P_M\onto M$ (resp. the cokernel of the embedding $M\inc I_M$). Let $D=\Hom_k(-,k)$ be the standard duality functor.

We structure the rest of the paper as follows. In Section \ref{sec:main_theorems}, we recall the definition of the delooping level and introduce our new invariants. We investigate the relationship among these invariants and prove our main Theorem \ref{thm:theorem1}. Section \ref{sec:trivial_extensions} contains our computation of the recent example in \cite{kershaw2023}, where $\ddell\Lambda=\Findim\Lambda^{\op}=1<\infty=\dell\Lambda$. This further shows the derived delooping level is a refinement over the delooping level. In Section \ref{sec:dell_vs_phidim}, we compare $\dell\Lambda$ and $\phidim\Lambda$ under a commonly satisfied condition and show they are both finite in that case. In Section \ref{sec:future_directions}, we discuss more questions about the findim conjecture and our new invariants. In particular, we show that the set of all modules with finite $\ddell\!$ forms a torsion-free class.

\textbf{Acknowledgments.} The second author would like to thank Gordana Todorov and Emre Sen for stimulating conversations about the contents of this paper. The authors also thank the anonymous referee for very helpful comments. The second author is grateful to the Simons Foundation for their support: Grant \# 686616.

\section{Old and New Invariants}
\label{sec:main_theorems}

We start with two well-known lemmas which will be used repeatedly throughout the section. We also include the proofs here for completeness.

\begin{lemma}
\label{lem:ext_simple}
If a module $X$ has finite injective dimension $n$, then there exists a simple module $S$ such that $\Ext_{\Lambda}^n(S,X)\neq 0$.
\end{lemma}
\begin{proof}
Let the minimal injective resolution of $X$ be
\[
0\to X\to I_0\to I_1 \to\cdots\to I_{n-1} \xrightarrow{f_n} I_n\to 0,
\]
and let $S$ be a simple direct summand of $\soc I_n$, the socle of $I_n$.

By definition, $\Ext_{\Lambda}^n(S,X) = \Hom_{\Lambda}(S, I_n) / (f_n)_*(\Hom_{\Lambda}(S,I_{n-1}))$. If  $\Ext_{\Lambda}^n(S,X) = 0$, then the inclusion $i$ from $S$ to $I_n$ factors as $i = S\xhookrightarrow{j} I_{n-1} \xrightarrow{f_n} I_n$. Since $I_{n-1}$ is injective, there is a map $g:I_n \to I_{n-1}$ such that $j=gi$. This implies $f_ngi=i$, so $f_ng: I_n\to I_n$ sends a copy of $I_S$ identically to itself. This holds for all simple summands of $\soc I_n$, so $0\to\ker f_n\to I_{n-1}\to I_n\to 0$ splits, contradicting the fact that the injective resolution is minimal.

In particular, the proof shows that all simple summands of $\soc I_n$ satisfy the lemma.
\end{proof}

\begin{lemma}
\label{lem:ext_syz}
For positive integers $n_1, n_2$ and modules $M, N$,
\[
\Ext_{\Lambda}^{n_1+n_2}(M,N)=\Ext_{\Lambda}^{n_1}(\Omega^{n_2}M,N).
\]

Dually,
\[
\Ext_{\Lambda}^{n_1+n_2}(M,N)=\Ext_{\Lambda}^{n_1}(M,\Sigma^{n_2}N).
\]
\end{lemma}
\begin{proof}
Since the dual statement is not used explicitly in the paper and the proofs of the two statements are similar, we only prove the first statement.

If $\pd M\leq n_2$, then $\Omega^{n_2}M$ is either projective or zero, so both $\Ext_{\Lambda}^{n_1+n_2}(M,N)$ and $\Ext_{\Lambda}^{n_1}(\Omega^{n_2}M,N)$ are zero.
Suppose $\pd M > n_2$. Then we have the following exact sequence
\[
0\to \Omega^{n_2} M \to P_{n_2-1} \to \cdots \to P_1 \to P_0 \to M \to 0.
\]
and $n_2$ short exact sequences
\begin{equation}
\label{eq:ext_ses}
0\to \Omega^{n}M \to P_{n-1} \to \Omega^{n-1}M \to 0
\end{equation}
for $n=1,2,\dots,n_2$.

We look at the following portions of long exact sequences induced by $\Ext(-,N)$:
\[
\Ext_{\Lambda}^{n_1}(P_{n_2-1},N) \to \Ext_{\Lambda}^{n_1}(\Omega^{n_2}M,N) \to \Ext_{\Lambda}^{n_1+1}(\Omega^{n_2-1}M,N) \to \Ext_{\Lambda}^{n_1+1}(P_{n_2-1},N),
\]
\[
\Ext_{\Lambda}^{n_1+1}(P_{n_2-2},N) \to \Ext_{\Lambda}^{n_1+1}(\Omega^{n_2-1}M,N) \to \Ext_{\Lambda}^{n_1+2}(\Omega^{n_2-2}M,N) \to \Ext_{\Lambda}^{n_1+2}(P_{n_2-2},N),
\]
and so on. Therefore, $$\Ext_{\Lambda}^{n_1}(\Omega^{n_2}M,N) \cong \Ext_{\Lambda}^{n_1+1}(\Omega^{n_2-1}M,N) \cong \Ext_{\Lambda}^{n_1+2}(\Omega^{n_2-2}M,N) \cong \cdots \cong \Ext_{\Lambda}^{n_1+n_2}(M,N). \qedhere $$
\end{proof}

We recall the definition of the delooping level and use the notation $M \dirsum N$ to mean $M$ is a \textbf{direct summand} of $N$. {\color{black}Let $\N$ denote the set of nonnegative integers. The infimum of the empty set is $+\infty$.}

\begin{definition}\cite{gelinas2022}
\label{def:dell}
Let $M$ be a $\Lambda$-module.
\begin{enumerate}
\item The delooping level of $M$ is
\[
\dell M = \inf \{n\in\N \mid \Omega^n M \dirsum \Omega^{n+1}N \text{ up to projective summands for some $N\in\smod\Lambda$} \}.
\]
\item The delooping level of $\Lambda$ is
\[
\dell\Lambda = \sup \{\dell S \mid \text{$S$ is a simple $\Lambda$-module} \}.
\]
\end{enumerate}
\end{definition}

\begin{remark}
The original definition of the delooping level is 
\begin{align*}
\dell M & = \inf \{n\in\N \mid \Omega^n M \text{ is a stable retract of } \Omega^{n+1}N \text{ for some $N\in\smod \Lambda$} \} \\
& = \inf \{n\in\N \mid \Omega^n M\xrightarrow{s} \Omega^{n+1} N \xrightarrow{\pi} \Omega^n M = \id\!_{\Omega^n M} \} \\
& = \inf \{n\in\N \mid \text{the short exact sequence } 0\to \Omega^n M \xrightarrow{s} \Omega^{n+1} N \to \coker{s} \to 0 \text{ splits} \},
\end{align*}
which is equivalent to Definition \ref{def:dell}. 
\end{remark}

A natural extension of Definition \ref{def:dell} is the $k$-delooping level.

\begin{definition}
\label{def:kdell}
Let $M$ be a $\Lambda$-module and $k\in \N\setminus \{0\}$. 
\begin{enumerate}
\item The \textbf{$k$-delooping level} of $M$ is
\[
\kdell M = \inf \{n\in\N \mid \Omega^n M \dirsum \Omega^{n+k}N \text{ up to projective summands for some $N\in\smod\Lambda$} \}.
\]
\item The \textbf{$k$-delooping level} of $\Lambda$ is
\[
\kdell\Lambda = \sup \{\kdell S \mid \text{$S$ is a simple $\Lambda$-module} \},
\]
\end{enumerate}
\end{definition}

{\color{black} Note that we allow $k=0$; this is for convenience when introducing $\ksubddell\!$ in Definition \ref{def:subddell}. Since $\Omega^0 M = M \dirsum \Omega^0 M$, we always have $0$-$\dell M=0$-$\dell\Lambda=0$. If $k>1$, $\kdell\Lambda$ is in general not a better upper bound for $\Findim\Lambda^{\op}$ than $\dell\Lambda$ since $\dell M \leq\kdell M$ for every module $M$ by definition. Therefore, the role of the $k$-delooping level in this paper is purely auxiliary when we introduce and prove results on the other two new invariants.

We believe that $\dell\Lambda\leq\kdell\Lambda$ is the most we can say about the relationship between them, in the sense that $\kdell\Lambda$ can be equal for all $k\neq 0$. We provide one of the many examples of this phenomenon, where $\kdell\Lambda=1$ for all $k\neq 0$.

\begin{example}[$k\neq 0$]
\label{ex:kedell_is_1_for_all_k}
Let $Q$ be the quiver
$
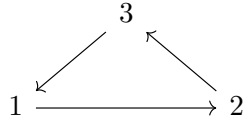
\begin{tikzcd} 
{} & 3 \arrow[dl] & {} \\
1 \arrow[rr] & {} & 2 \arrow[lu]
\end{tikzcd}$ such that the indecomposable projective modules of $KQ/I$ are
$\begin{matrix}
1 \\ 2 \\ 3
\end{matrix}$,
$\begin{matrix}
2 \\ 3
\end{matrix}$, and
$\begin{matrix}
3 \\ 1 \\ 2
\end{matrix}$. Then it is clear that $\kdell S_1=1$ for all $k$ since $S_1$ cannot be a syzygy and $\Omega S_1$ is projective.

From the projective resolutions of $S_2$ and $S_3$, we notice that $S_3=\Omega^{2n} S_3$ and $S_3=\Omega^{2n+1} S_2$ for all $n\in \N$, so $\kdell S_3=0$ for all $k$.

For the simple module at 2, we have $\dell S_2=0$ since $S_2=\Omega \left(\begin{matrix} 3 \\ 1 \end{matrix}\right) $, but $S_2$ is not a second syzygy. So $\kdell S_2=1$ for $k>1$ since $\Omega S_2=S_3$. Therefore, $\kdell\Lambda=1$ for all $k$.
\end{example}
}

\textbf{Terminology.} It is clear that if a module $M$ occurs as (direct summand of) syzygies of all orders, \textit{i.e.}, $\kdell M=0$ for all $k$, then whenever $M$ shows up as a summand of $\Omega^n S$ for some $n$, we can ignore $M$ when calculating $\kdell S$. In that case, we say the module $M$ is \textbf{infinitely deloopable}. For example, $S_3$ in Example \ref{ex:kedell_is_1_for_all_k} is infinitely deloopable, and when we get $\Omega S_2=S_3$, we immediately know $\kdell S_2\leq 1$ for all $k$.

{\color{black}
Next, we define the effective delooping level $\edell\!$ and show that it is indeed a reinterpretation of Findim. This is inspired by the original proof of $\Findim\Lambda^{\op}\leq \dell\Lambda$ in \cite{gelinas2022}, wherein we prove $\dell S\leq n$ for all simple modules $S$ implies $\id X\leq n$ for all modules $X$.

\begin{proof}[Proof of $\Findim\Lambda^{\op}\leq\dell\Lambda$.]
Suppose $\dell\Lambda=n$ so that for every simple module $S$, there exists a module $N_S$ such that $\Omega^n S\dirsum \Omega^{n+1} N_S$. Suppose for a contradiction that there is a module $X$ with $\id X=n+k'>n$. Then by Lemma \ref{lem:ext_simple}, there exists a simple module $S$ such that $\Ext_{\Lambda}^{n+k'}(S,X)\neq 0$. On the other hand, using Lemma \ref{lem:ext_syz}, we get
\[
\Ext_{\Lambda}^{n+k'}(S,X) \cong \Ext_{\Lambda}^{k'}(\Omega^n S,x) \dirsum \Ext_{\Lambda}^{k'}(\Omega^{n+1} N_S, X) \cong \Ext_{\Lambda}^{n+k'+1}(N_S,X)=0,
\]
which is a contradiction.
\end{proof}

The definition of the effective delooping level aims to pinpoint and generalize the key idea of this proof.
}

\begin{definition}
\label{def:edell}
Let $M$ be a $\Lambda$-module and $k,k'$ be positive integers. 
\begin{enumerate}
\item The \textbf{$k$-effective delooping level} of $M$ is
\[
\kedell M = \inf \{n\in\N \mid \id X = n+k-1+k'>n+k-1 \text{ implies } \Ext_{\Lambda}^{n+k'}(M,X)=0 \}.
\]
\item The \textbf{$k$-effective delooping level} of $\Lambda$ is
\[
\kedell \Lambda = \sup \{\kedell S\mid S \text{ is simple} \}
\]
\end{enumerate}
We drop the $k$ when it is 1. Thus, $\edell M:=1\text-\edell M$ and the effective delooping level of $\Lambda$ is defined to be the $1$-effective delooping level of $\Lambda$.
\end{definition}

It turns out this definition perfectly captures the big finitistic dimension when $k=1$.

{\color{black}
\begin{proposition}
\label{prop:edell_vs_findim}
$\Findim\Lambda^{\op} = \edell\Lambda$.
\end{proposition}

\begin{proof}
If $\edell\Lambda=\infty$, or equivalently, $\edell S=\infty$ for some simple module $S$, then by definition, for all $n\in\N$, there is some module $X$ with $\id X=n+k'>n$ such that $\Ext_{\Lambda}^{n+k'}(S,X)\neq 0$. This trivially guarantees there are modules with arbitrarily high injective dimension.

Now suppose $\edell \Lambda=n<\infty$. Then there exists a simple module $S$ such that $\edell S=n$, \textit{i.e.}, for every module $X$ such that $\id X=n+k'>n$, $\Ext_{\Lambda}^{n+k'}(S,X) = 0$. It is immediate that there can be no module $X$ with $\id X>n$ as it would contradict Lemma \ref{lem:ext_simple}. On the other hand, if every module $X$ with finite injective dimension has $\id X<n$, then $\edell \Lambda\leq n-1$ by definition since it is vacuously true. Therefore, there must exist a module $X$ with $\id X\geq n$. Since $\id X>n$ is impossible, $\Findim\Lambda^{\op}=n$.
\end{proof}

\begin{remark}
\label{rmk:kedell_gives_lower_bound_of_Findim}
If $\kedell\Lambda=n$ for $k>0$, we may use the last argument in the previous proof to show the lower bound $\Findim\Lambda^{\op}\geq n+k-1$. However, this is not useful in practice when $k>1$ since in that case, $\kedell\Lambda$ is always harder to compute than $\Findim\Lambda^{\op}$.
\end{remark}

We can easily compare $\kedell M$ and $\kdell M$. In fact, this will be a consequence of certain later theorems in the paper, but we present a proof here first.

\begin{lemma}
\label{lem:kedell_vs_kdell}
$\kedell M \leq \kdell M$.
\end{lemma}
\begin{proof}
Assume $\kdell M=n<\infty$ so that there exists a module $N$ such that $\Omega^n M\dirsum \Omega^{n+k} N$. To show $\kedell M\leq n$, we pick any module $X$ such that $\id X=n+k-1+k'>n+k-1$ and want to show that $\Ext_{\Lambda}^{n+k'}(M,X)=0$. By Lemma \ref{lem:ext_syz},
\[
\Ext_{\Lambda}^{n+k'}(M,X) \cong \Ext_{\Lambda}^{k'}(\Omega^n M,X) \dirsum \Ext_{\Lambda}^{k'}(\Omega^{n+k} N,X) \cong \Ext_{\Lambda}^{n+k+k'}(N,X) = 0.
\]
\end{proof}
}

As expected, the effective delooping level is difficult to compute in general as a reinterpretation of the big finitistic dimension since we need to consider all $X$ with high injective dimension and compute extensions. Nonetheless, we use this reinterpretation to improve the upper bound of the Findim by establishing relationships between the effective delooping level and our other new invariants.

Coming back to Gelinas' question about whether $\dell\Lambda=\Findim\Lambda^{\op}$, we want to motivate our new definitions with a counterexample and demonstrate the shortcoming of concentrating on the delooping level of simple modules instead of all modules. This example is particularly interesting because the algebra is monomial, one of the nicest and well-understood classes of algebras to consider and one that guarantees the finiteness of both findim and Findim.

\begin{example}
\label{ex:monomial_example}
(Revisited in Examples \ref{ex:monomial example revisited subddell} and \ref{ex:monomial example revisited ddell})
Let $Q$ be the quiver

\begin{center}
\begin{tikzcd}
1 \arrow[r,shift left,"\alpha_1"] \arrow[r,leftarrow,shift right,"\alpha_2",swap] & 2 \arrow[out=120,in=60,loop,looseness=3,"\beta"] \arrow[r,"\gamma"] & 3 \arrow[r,"\delta"] & 4 \arrow[r,"\epsilon"] & 5
\end{tikzcd}
\end{center}
with relations $\alpha_1\alpha_2, \alpha_1\beta, \alpha_1\gamma\delta, \beta^2, \beta\gamma, \beta\alpha_2, \alpha_2\alpha_1\gamma$. The indecomposable projective modules of the path algebra $\Lambda=KQ/I$, where $I$ is the ideal generated by the above relations, are

\begin{center}
$\begin{matrix}
1 \\ 2 \\ 3
\end{matrix}
\qquad
\begin{matrix}
& 2 & \\
2 & 1 & 3 \\
& 2 & 4 \\
& & 5
\end{matrix}
\qquad
\begin{matrix}
3 \\ 4 \\ 5
\end{matrix}
\qquad
\begin{matrix}
4 \\ 5
\end{matrix}
\qquad 
5$.
\end{center}

Since the simple modules at 2, 3, and 5 are summands of the socle of some projectives and 1, 4 are not, $\dell S_2 = \dell S_3 =\dell S_5 = 0$ and $\dell S_1, \dell S_4 \neq 0$. Since $\Omega S_4 = S_5$ is projective, $\dell S_4 = 1$. 

It is known that second syzygies of modules over a monomial algebra are direct sums of $q\Lambda$ where $q$ is a path of length $\geq 1$ \cite[Theorem I]{huisgen1991predicting}. Note that $\Omega S_1 = \begin{matrix} 2 \\ 3 \end{matrix}$ is one of such $q\Lambda$ only when $q$ is $\alpha_1$, so the only way in which $\begin{matrix} 2 \\ 3 \end{matrix}$ is a second or higher syzygy is when $S_1$ is a syzygy. This is clearly not the case, so $\dell S_1\neq 1$. 

Now, we compute that $\Omega^2 S_1 = \Omega\, \left(\begin{matrix} 2 \\ 3\end{matrix}\right) = S_2 \oplus \begin{matrix} 1 \\ 2 \end{matrix} \oplus \begin{matrix} 4 \\ 5 \end{matrix}$. The module $\begin{matrix} 4 \\ 5 \end{matrix}$ is projective. The modules $S_2$ and $\begin{matrix} 1 \\ 2 \end{matrix}$ are direct summands of $\Omega^k S_2$ for any $k>0$, so they are infinitely deloopable. This implies $\dell S_1 = 2$. Therefore, we get $\dell\Lambda = 2$.

{\color{black}
We claim that $\Findim\Lambda^{\op}=1$. There are several ways to prove this, and we present one using another useful theorem in \cite{huisgen1991predicting}. The quiver of the opposite algebra $\Lambda^{\op}$ is
\begin{center}
\begin{tikzcd}
1 \arrow[r,shift left,"\alpha_1", leftarrow] \arrow[r,shift right,"\alpha_2",swap] & 2 \arrow[out=120,in=60,loop,looseness=3,"\beta"] \arrow[r,"\gamma", leftarrow] & 3 \arrow[r,"\delta",leftarrow] & 4 \arrow[r,"\epsilon",leftarrow] & 5,
\end{tikzcd}
\end{center}
and the projective modules of $\Lambda^{\op}$ are
\begingroup % keep the change local
\setlength\arraycolsep{1pt}
\begin{center}
$\begin{matrix}
1 \\ 2
\end{matrix}
\qquad
\begin{matrix}
& 2 & \\
1 &  & 2 \\
2 & & \\
\end{matrix}
\qquad
\begin{matrix}
3 \\ 2 \\ 1
\end{matrix}
\qquad
\begin{matrix}
4 \\ 3 \\ 2
\end{matrix}
\qquad 
\begin{matrix}
5 \\ 4 \\ 3 \\ 2
\end{matrix}.$
\end{center}
\endgroup

We refer to a useful definition of the number $s$ in Corollary II of \cite{huisgen1991predicting} and compute
\[
s=\sup \{ \pd q\Lambda^{\op} \mid \text{$q$ is a path of length $\geq 1$ and $\pd q\Lambda^{\op}<\infty$}\}.
\]

\color{black}
Then by verifying the conditions in Theorem VI and Remark 10 in \cite{huisgen1991predicting}, which we compile as Theorem \ref{thm:monomial_alg_Findim}, we will obtain $\findim\Lambda^{\op} = \Findim\Lambda^{\op} = s+1 = 1$.
\color{black}
\begin{theorem}
\label{thm:monomial_alg_Findim}
Let $s\geq 0$ and factor all paths $q_1, \dots, q_t$ of lengths $\geq 1$ such that $\pd q_j\Lambda^{\op} = s$ into the form $q_j=p_jr_j$ for $1\leq j\leq t$, where $r_j$ is a path of length 1. In particular, if $q_j$ is an arrow, $r_j=q_j$ and $p_j$ is $e_{s(r_j)}$, the trivial path at $s(r_j)$. If for all $j=1,\dots,t$, 
\begin{itemize}
\item $wp_j=0$ for all paths $w$ in $\mathrm{l.ann}\, q_j$, the left annihilator of $q_j$, or
\item for each set of paths $A\subseteq \mathrm{l.ann}\, q_j$ with $Ap_j\neq 0$, there exists a summand $\beta\Lambda^{\op}$ of $\mathrm{r.ann}\, A$, the right annihilator of $A$, such that $\beta$ is an arrow with $\pd \beta\Lambda^{\op} = \infty$,
\end{itemize}
then $\findim\Lambda^{\op} = \Findim\Lambda^{\op} = s+1$.
\end{theorem}

Note that we can also write
$$\mathrm{l.ann}\, q_j = \oplus \{\Lambda^{\op} u \mid \text{$u$ is a path of length $\geq 1$ which is left minimal with respect to $uq_j=0$}\},$$
where the left minimality condition means that $u'q_j\neq 0$ whenever $u=pu'$ for some arrow $p$.

Now, we will verify either of the two conditions in Theorem \ref{thm:monomial_alg_Findim} to show $\Findim\Lambda^{\op}=1$. A quick computation shows $s=0$ for $\Lambda^{\op}$ and the paths $q$ making $\pd q\Lambda^{\op}=0$ are arrows $\alpha_1$ and $\epsilon$. We factor them as $\alpha_1 = e_2\alpha_1$ and $\epsilon = e_5\epsilon$, and find 
\[
\mathrm{l.ann}\, \alpha_1 = \oplus\{\Lambda^{\op} u \mid u=\alpha_1, \beta, \alpha_1\alpha_2, \delta\gamma, \delta, \epsilon\},
\]
\[
\mathrm{l.ann}\, \epsilon = \oplus\{\Lambda^{\op} u \mid u=\alpha_1, \alpha_2, \beta, \gamma, \delta, \epsilon\}.
\]

We use the first condition in Theorem \ref{thm:monomial_alg_Findim} for $\epsilon$. It is clear that $we_5=0$ for all $w\in\{\alpha_1, \alpha_2, \beta, \gamma, \delta, \epsilon\}$ since vertex 5 is a source. For $\alpha_1$, three minimal paths $\alpha_1\alpha_2$, $\beta$, and $\delta\gamma$ terminate at vertex 2, so we use the second condition. The arrow $p=\alpha_1\alpha_2$ satisfies the condition, as $\alpha_1\alpha_2$ is a right annihilator of all of $\alpha_1, \alpha_2, \beta, \gamma, \delta, \epsilon$ and $\pd \alpha_1\alpha_2\Lambda^{\op} = \pd S_2 = \infty$. Therefore, we conclude $\Findim\Lambda^{\op}=1$ and we have created the gap $\dell\Lambda-\Findim\Lambda^{\op}=1$.
}
\end{example}

Our new definitions aim to shorten or close the gap between $\dell\Lambda$ and $\Findim\Lambda^{\op}$. The new names sub-derived and derived delooping levels come from rotating triangles in the derived category and its shift functor being the syzygy functor. To wit, we have the following lemma.

\begin{lemma}
\label{lem:rotating_ses}
If $0\to A \xrightarrow{f} B \xrightarrow{g} C\to 0$ is a short exact sequence in $\mod\Lambda$, then there is also an exact sequence
\begin{equation}
\label{eq:lem_rotating_ses}
0 \to \Omega C \xrightarrow{f'} A\oplus P_C \xrightarrow{g'} B\to 0,
\end{equation}
where $P_C$ is the projective cover of $C$.

More generally, if $0\to M_n \to \cdots \to M_1 \to M_0 \to 0$ is an exact sequence in $\mod\Lambda$, then there is also an exact sequence
\begin{equation}
\label{eq:rotating_ses_long}
0\to \Omega^k M_n \oplus P_n \cdots \to \Omega^k M_1 \oplus P_1 \to \Omega^k M_0 \to 0
\end{equation}
for all $k\in\Z_{>0}$ and some projectives $P_1,\dots,P_n$. Alternatively, we drop the projectives and write \eqref{eq:rotating_ses_long} as
\[
0\to \Omega^k M_n \cdots \to \Omega^k M_1 \to \Omega^k M_0 \to 0
\]
in $\smod\Lambda$, since all of our computations are mod projectives.
\end{lemma}
\begin{proof}
The proof of the general statement \eqref{eq:rotating_ses_long} can be obtained from repeatedly applying \eqref{eq:lem_rotating_ses}, so we only prove \eqref{eq:lem_rotating_ses}. {\color{black} Note that if any of $A, B, C$ is zero, then \eqref{eq:lem_rotating_ses} holds. If $C$ is projective (so that $\Omega C=0$), then \eqref{eq:lem_rotating_ses} holds since $0\to A\to B\to C\to 0$ splits.}

Now assume none of $A, B, C, \Omega C$ is zero. Consider the following commutative diagram.

\begin{center}
\begin{tikzcd}
0 \arrow[r] & A \arrow[ddrr, leftarrow, dashed, "\bar{f}"] \arrow[r,"f"] & B \arrow[dr, leftarrow, dashed, "\bar{g}"] \arrow[r, "g"] & C \arrow[r] \arrow[d, leftarrow, "\pi"] & 0 \\
{} & {} & {} & P_C \arrow[d, leftarrow, "i"] & {} \\
{} & {} & {} & \Omega C & {}
\end{tikzcd}
\end{center}

Since $g:B\to C$ is surjective and there is a surjection $\pi: P_C \to C$ from the projective cover of $C$ to $C$, we obtain a map $\bar{g}: P_C\to B$ such that $g\bar{g}=\pi$. We can define $g':A\oplus P_C\to B$ by 
\[
g'(a,p)=f(a)-\bar{g}(p).
\]
We also know that $\im\bar{g}i \subseteq \ker g=\im f$, so the map $\bar{f}(x)=f^{-1}(\bar{g}i(x))$ from $\Omega C$ to $A$ is well-defined. Naturally, define $f':\Omega C\to A\oplus P_C$ by 
\[
f'(x)=(\bar{f}(x),i(x)).
\]

Now it is straightforward to show that \eqref{eq:lem_rotating_ses} is exact.
\end{proof}

We continue with the definition of the sub-derived delooping level.

{\color{black}
\begin{definition}
\label{def:subddell}
Let $M$ be a $\Lambda$-module.
\begin{enumerate}
\item For $k\in\Z_{>0}$, the $k$-sub-derived delooping level of $M$ is
\begin{align*}
\ksubddell M = \inf\{m \mid & \, \exists n\leq k, \text{and an exact sequence in $\mod\Lambda$ of the form} \\
& 0\to M \to D_0 \to D_1 \to \cdots \to D_{n-1} \to D_n \to 0, \\
& \text{where $(k-i)$-$\dell D_i \leq i+m$, } i=0,1,\dots,n \},
\end{align*}
where we drop the $k$ when it is 1 and write $\subddell M$ instead of $1\text{-}\subddell M$.
We say $\ksubddell M$ \textbf{is equal to $m$ using $n$}.
\item For $k\in\Z_{>0}$, the $k$\text{-}sub-derived delooping level of $\Lambda$ is
\[
\ksubddell \Lambda = \sup \{\ksubddell S \mid S \text{ is simple} \},
\]
where we drop the $k$ when it is 1. Thus, the sub-derived delooping level of $\Lambda$ is defined to be the 1-sub-derived delooping level of $\Lambda$.
\end{enumerate}
\end{definition}

This definition is quite involved, so we make three relevant remarks. The first compares $\ksubddell M$ and $\kdell M$. The second simplifies the definition when $k=1$. The third shows $\ksubddell M$ increases as $k$ increases.

\begin{remark}
\label{rmk:subddell_vs_kdell}
For every module $M$, we can always take the short exact sequence $0\to M \to M\oplus \Lambda \to \Lambda\to 0$. This shows $\ksubddell M\leq \kdell M$.
\end{remark}

\begin{remark}
\label{rmk:subddell}
If $n\neq 0$, the definition of $\ksubddell M$ applied to $k=1$ is
\begin{align}
\label{eq:1subddell-a}
1\text{-}\subddell M = \inf\{m \mid & \, \text{there exists an exact sequence in $\mod\Lambda$ of the form} \\
& 0\to M \to D_0 \to D_1 \to 0, \text{ where $(1-i)$-$\dell D_i\leq i+m, i=0,1$} \nonumber\}.
\end{align}

In this particular case, the 0-delooping level imposes nothing on $D_1$ and we only need the injection $M\xhookrightarrow{} D_0$. The smallest $m$ that can be achieved is $\inf\{\dell D_0 \mid M\xhookrightarrow{} D_0\}$. %If $n=0$, the smallest possible $m$ that can be achieved is $\dell M \leq \inf\{\dell D_0 \mid M\xhookrightarrow{} D_0\}$, 
So, we may define the $1$-$\subddell M$ more easily as
\begin{equation}
\label{eq:1subddell}
1\text{-}\subddell M = \inf\{\dell N \mid M \inc N\}.
\end{equation}

When we refer to $\subddell\!$ in the future, we will use this definition \eqref{eq:1subddell}.
\end{remark}

\begin{remark}
\label{rmk:ksubddell is increasing}
For $k>0$, if $(k+1)$-$\subddell M$ is $m$ using $n$ and the exact sequence
\begin{equation}
\label{eq:ksubddell is increasing}
0\to M\to D_0 \to \cdots \to D_n \to 0,
\end{equation}
where $(k+1-i)$-$\dell D_i\leq m+i$ for $i=0,1,\dots,n$, then we claim that $\ksubddell M\leq m$. There are two cases: $n<k+1$ and $n=k+1$.

If $n<k+1$, we can use the same exact sequence \eqref{eq:ksubddell is increasing} and have $(k-i)$-$\dell D_i \leq (k+1-i)$-$\dell D_i\leq m+i$ for $i=0,1,\dots,n$, so $\ksubddell M\leq m$.

If $n=k+1$, we simply truncate \eqref{eq:ksubddell is increasing} to have $0\to M \to D_0 \to \cdots \to D_{k-1} \to Coker\to 0$, where $(k-i)$-$\dell D_i \leq m+i$ still holds and we do not need any condition on the cokernel.
\end{remark}

The sub-derived delooping level is strictly better than the original delooping level in Example \ref{ex:monomial_example}.

\begin{example}[Example \ref{ex:monomial_example} revisited]
\label{ex:monomial example revisited subddell}
By observing the projective resolution of $S_2$, we know both $S_2$ and $\begin{matrix} 1 \\ 2 \end{matrix}$ are infinitely deloopable. Therefore, the embedding $S_1\xhookrightarrow{} \begin{matrix} 2 \\ 1 \end{matrix}$ implies $\subddell\Lambda=1<2=\dell\Lambda$.

For every $k>1$, we can use the same exact sequence $0\to 1\to \begin{matrix} 2 \\ 1 \end{matrix} \to 2 \to 0$ to show $\ksubddell\Lambda=1$ for all $k$.
\end{example}

The sub-derived delooping level also bounds the opposite Findim.

\begin{theorem}
\label{thm:subddell}
$\Findim\Lambda^{\op} \leq \subddell\Lambda \leq \dell\Lambda$.
\end{theorem}

We prove the theorem using induction, following the next two lemmas. In fact, we will prove the more general statement
\[
\kedell\Lambda \leq \ksubddell\Lambda \leq \kdell\Lambda,
\]
where the theorem is the special case when $k=1$.
}

\begin{lemma}[$\subddell\!$ Base Case]
\label{lem:subddell_base_case}

If the map of modules $f:A\to B$ is injective and $\dell B\leq n$, then $\edell A\leq n$.
\end{lemma}

\begin{proof}
If $f$ is also surjective, then $A\cong B$, so the Lemma follows since $\edell A\leq \dell A$.

Now assume $f$ is not surjective so that $\coker f\neq 0$. Let $X$ be a module with $\id X=n+k'>n$. Then we can rotate the short exact sequence $0\to A\to B\to \coker f\to 0$ by Lemma \ref{lem:rotating_ses} to get another short exact sequence
\[
0 \to \Omega^{n+1} \coker f \to \Omega^{n} A \to \Omega^{n} B \to 0,
\]

Apply $\Ext^{k'}_{\Lambda}(-,X)$ to obtain the following exact sequence
\[
\Ext_{\Lambda}^{k'}(\Omega^{n+1}\coker f,X) \xleftarrow{} \Ext_{\Lambda}^{k'}(\Omega^{n}A,X) \xleftarrow{} \Ext_{\Lambda}^{k'}(\Omega^{n}B,X).
\]

The first term is isomorphic to $\Ext_{\Lambda}^{n+k'+1}(\coker f,X)=0$. Since $\dell B\leq n$, there exists some module $N_B$ such that the last term is a direct summand of $\Ext_{\Lambda}^{k'}(\Omega^{n+1} N_B, X)\cong \Ext_{\Lambda}^{n+k'+1}(N_B, X)=0$. Therefore, $\Ext_{\Lambda}^{k'}(\Omega^{n}A,X)\cong \Ext_{\Lambda}^{n+k'}(A,X)=0$.
\end{proof}

{\color{black}
\begin{lemma}[$\subddell\!$ Inductive Step]
\label{lem:subddell_inductive_case}
If $0\to A\to B\to C\to 0$ is a short exact sequence in $\mod\Lambda$ and for $k>1$,
\begin{itemize}
\item $\kdell B\leq n$
\item $(k-1)$-$\edell C\leq n+1$,
\end{itemize}
then $\kedell A \leq n$.
\end{lemma}

\begin{proof}
By applying Lemma \ref{lem:rotating_ses} repeatedly, we get the short exact sequence
\begin{equation}
\label{eq:SES in subddell proof}
0\to \Omega^{n+1} C \to \Omega^n A \to \Omega^n B \to 0.
\end{equation}

To show $\kedell A\leq n$, take any module $X$ such that $\id X = n+k-1+k'>n+k-1$. Apply $\Ext_{\Lambda}^{k'}(-,X)$ to \eqref{eq:SES in subddell proof} to get another exact sequence
\begin{equation}
\label{eq:Ext SES in subddell proof}
\Ext_{\Lambda}^{k'}(\Omega^{n+1}C,X) \xleftarrow{} \Ext_{\Lambda}^{k'}(\Omega^n A,X) \xleftarrow{}  \Ext_{\Lambda}^{k'}(\Omega^n B,X).
\end{equation}

Since $\kdell B\leq n$, there exists a module $N_B$ such that $\Omega^n B\dirsum \Omega^{n+k} N_B$, so the rightmost term $\Ext_{\Lambda}^{k'}(\Omega^n B,X)$ of \eqref{eq:Ext SES in subddell proof} is a direct summand of $\Ext_{\Lambda}^{k'}(\Omega^{n+k} N_B, X) \cong \Ext_{\Lambda}^{n+k+k'}(N_B, X)=0$.

Since $(k-1)$-$\edell C\leq n+1$, for the $X$ we chose with $\id X = n+k-1+k'>n+k-1$, we have $\Ext_{\Lambda}^{n+1+k'}(C, X)=0$, but it is also isomorphic to $\Ext_{\Lambda}^{k'}(\Omega^{n+1}C,X)$, the leftmost term of \eqref{eq:Ext SES in subddell proof}. Therefore, $\Ext_{\Lambda}^{n+k'}(A, X)\cong \Ext_{\Lambda}^{k'}(\Omega^n A,X)=0$, so $\kedell A \leq n$.
\end{proof}

}

{\color{black}

\begin{proof}[Proof of Theorem \ref{thm:subddell}] Suppose $\ksubddell\Lambda=m<\infty$. Then there is a simple module $S$ satisfying $\ksubddell S = m$ using $n$. We will show $\kedell S\leq m$. Consider the exact sequence used to determine $\ksubddell S$:
\begin{equation}
0 \to S \overset{f_0}{\to} D_0 \overset{f_1}{\to} D_1 \to \cdots \to D_{n-1} \overset{f_n}{\to} D_n \to 0,
\end{equation}
where $(k-i)$-$\dell D_i \leq i+m$, $i=0,1,\dots,n$.

If $n<k$, we only need Lemma \ref{lem:subddell_inductive_case}. Starting with the short exact sequence
\[
0\to \coker f_{n-2} \to D_{n-1} \to D_n \to 0\]
with 
\begin{itemize}
\item $(k-n+1)$-$\dell D_{n-1} \leq m+n-1$
\item $(k-n)$-$\edell D_n \leq (k-n)$-$\dell D_n\leq m+n$
\end{itemize}
we apply Lemma \ref{lem:subddell_inductive_case} to obtain $(k-n+1)$-$\edell \coker f_{n-2} \leq m+n-1$.

Proceed inductively until the short exact sequence $0 \to \coker f_0 \to D_1 \to \coker f_1\to 0$ to obtain $(k-1)$-$\edell \coker f_0\leq m+1$. Lastly, apply Lemma \ref{lem:subddell_inductive_case} again to $0\to S\to D_0 \to \coker f_0 \to 0$ to conclude $\kedell S\leq m$.

If $n=k$, then there is no condition on $D_n$ and $\dell D_{n-1}\leq m+n-1$. So we can apply Lemma \ref{lem:subddell_base_case} to $\coker f_{n-2} \xhookrightarrow{} D_{n-1}$ to obtain $\edell \coker f_{n-2}\leq m+n-1$. Similar to the case of $n<k$, we apply Lemma \ref{lem:subddell_inductive_case} $(k-1)$ more times to conclude $\kedell S\leq m$.

The same argument works for all simple modules $S$, so $\kedell \Lambda\leq \ksubddell\Lambda$ and the theorem is the special case when $k=1$. \end{proof}

It is natural to consider the definition and theorem dual to those of the sub-derived delooping level. Hence, we introduce the derived delooping level, the important definition of the paper.
}

{\color{black}
\begin{definition}
\label{def:ddell}
Let $M$ be a $\Lambda$-module.
\begin{enumerate}
\item The \textbf{$k$-derived delooping level} of $M$ is
\begin{align*}
\kddell M = \inf \{m\in\N \mid & \,\exists n\leq m \text{ and an exact sequence in $\mod\Lambda$ of the form} \\
& \,\, 0 \to C_n \to C_{n-1} \to \cdots \to C_1 \to C_0 \to M \to 0, \\
& \text{ where $(i+k)$-$\dell C_i\leq m-i$, } i=0,1,\dots,n  \},
\end{align*}
where we drop the $k$ when it is 1 and write $\ddell M$ instead of $1\text{-}\ddell M$. We say $\kddell M$ is \textbf{equal to $m$ using $n$}.
\item The \textbf{$k$-derived delooping level} of $\Lambda$ is
\[
\kddell \Lambda = \sup \{\kddell S\mid S \text{ is simple} \}
\]
\end{enumerate}
We drop the $k$ when it is 1. Thus, $\ddell M:=1\text-\ddell M$ and the derived delooping level of $\Lambda$ is defined to be the $1$-derived delooping level of $\Lambda$.
\end{definition}

In order to better understand this new definition and its connection to other invariants, we make three remarks like when we introduced $\subddell\!$.

\begin{remark}
\label{rmk:ddell_less_than_dell}
Again, we immediately see that $\kddell \Lambda \leq \kdell \Lambda$ from the definitions. If $\kdell S=n<\infty$ for some simple module $S$, then we can truncate the projective resolution of $S$ at $P_n$ and form the exact sequence
\[
0\to \Omega^n S \to P_{n-1} \to \cdots \to P_1 \to P_0 \to S \to 0.
\]
We see immediately that $\kddell S\leq n$.
\end{remark}
}

\begin{remark}
\label{rmk:finite_k+1_ddell_implies_finite_k_ddell}
In general, for any $k>0$ and any module $M$, if $(k+1)$-$\ddell M=m$ using $n$ is finite, then $\kddell M\leq m$. This is because the exact sequence
\[
0\to C_n \to \cdots \to C_0 \to M \to 0
\]
used to determine $(k+1)$-$\ddell M$ has $(i+k+1)$-$\dell C_i \leq m-i$ for $i=0,1,\dots,n$. This means $(i+k)$-$\dell C_i\leq m-i$ for $i=0,1,\dots,n$, so $\kddell M$ is at most $m$.
\end{remark}

\begin{remark}
\label{rmk:ddell_subddell_asymmetry}
The asymmetry in the definitions of the derived and sub-derived delooping levels is necessary. It might make sense to define the sub-derived delooping level as
\begin{align*}
\subddell M = \inf\{n\in\N & \mid M \text{ fits in an exact sequence in $\mod\Lambda$ of the form} \\
& \,\, 0 \to M \to D_0 \to D_1 \to \cdots \to D_{n-1} \to D_n \to 0, \\
& \text{ where $(i+1)$-$\dell D_i\leq n-i$} \}.
\end{align*}
However, we saw in Lemma \ref{lem:subddell_base_case} that the exact sequence $0\to M\to D_0$ is already sufficient in showing $\Findim\Lambda^{\op}\leq \subddell\Lambda$. Similarly, the derived delooping level cannot be defined concisely. That is, it is insufficient to define $\ddell M = \inf\{ \dell N \mid N \onto M\}$ and $\ddell\Lambda = \sup \{\ddell S\mid S \text{ is simple} \}$ and show $\Findim\Lambda^{\op}\leq \ddell\Lambda$. Otherwise, we could just take $N$ to be the projective cover and the delooping level of any projective is zero.
\end{remark}

The derived delooping level is also strictly better than the delooping level in Example \ref{ex:monomial_example}.

\begin{example}[Example \ref{ex:monomial_example} revisited]
\label{ex:monomial example revisited ddell}
We easily find that $\ddell S_1 = 1$ due to the short exact sequence $0\to S_2 \to \begin{matrix} 1 \\ 2 \end{matrix} \to S_1 \to 0$ and both $S_2$ and $\begin{matrix} 1 \\ 2 \end{matrix}$ being infinitely deloopable, which was pointed out in Example \ref{ex:monomial example revisited subddell}. For other values of $k$, the same short exact sequence shows $\kddell\Lambda=1$ using 1 for all $k$.

Therefore, for the monomial algebra $\Lambda$ defined in Example \ref{ex:monomial_example},
\[
\subddell\Lambda=\ddell\Lambda=\Findim\Lambda^{\op}=1<\dell\Lambda=2.
\]
\end{example}

A main feature of $\subddell\!$ and $\ddell\!$ is that they are usually easier to compute, as we are allowed to explore all modules that have a map to or from simple modules and more. In fact, we conjecture that $\ddell\Lambda = \Findim\Lambda^{\op}$ for all $\Lambda$. The statement $\subddell\Lambda=\Findim\Lambda^{\op}$ is not true, as we will show the example from \cite{kershaw2023} is a counterexample in Section \ref{sec:trivial_extensions}.

To finish the discussion of our new invariants, we prove $\ddell\Lambda$ is another upper bound for $\Findim\Lambda^{\op}$.

\begin{theorem}
\label{thm:ddell}
$\Findim\Lambda^{\op} \leq \ddell\Lambda \leq \dell\Lambda$.
\end{theorem}

Like the case for $\subddell\!$, we prove the more general statement
\[
\kedell\Lambda \leq \kddell\Lambda \leq \kdell\Lambda,
\]
so that the theorem is the special case when $k=1$.

\begin{lemma}[$\ddell\!$ Base Case]
\label{lem:v1_induction}
If $0\to A\to B\to C\to 0$ is a short exact sequence in $\mod\Lambda$ and for $k>1$,
\begin{itemize}
\item $\kdell A\leq n-1$
\item $(k-1)$-$\dell B\leq n$,
\end{itemize}
then $(k-1)$-$\edell C\leq n$.
\end{lemma}

\begin{lemma}[$\ddell\!$ Inductive Step]
\label{lem:v2_induction}
If $0\to A\to B\to C\to 0$ is a short exact sequence in $\mod\Lambda$ and for $k>1$,
\begin{itemize}
\item $\kedell A\leq n-1$
\item $(k-1)$-$\dell B\leq n$,
\end{itemize}
then $(k-1)$-$\edell C\leq n$.
\end{lemma}

Note that Lemma \ref{lem:v1_induction} implies Lemma \ref{lem:v2_induction} due to Lemma \ref{lem:kedell_vs_kdell}, so we only prove Lemma \ref{lem:v1_induction}. 

\begin{proof}[Proof of Lemma \ref{lem:v1_induction}]
Applying Lemma \ref{lem:rotating_ses} repeatedly, we get the short exact sequence
\begin{equation}
\label{eq:rotated_ses}
0\to \Omega^n B \to \Omega^n C \to \Omega^{n-1} A \to 0.
\end{equation}

To show $(k-1)$-$\edell C\leq n$, we take any module $X$ with $\id X=n+k-2+k'>n+k-2$ and apply the contravariant functor $\Ext_{\Lambda}^{k'}(-,X)$ to \eqref{eq:rotated_ses} to get the exact sequence
\begin{equation}
\label{eq:Ext_sequence}
\Ext_{\Lambda}^{k'}(\Omega^n B,X) \leftarrow \Ext_{\Lambda}^{k'}(\Omega^n C,X) \leftarrow \Ext_{\Lambda}^{k'}(\Omega^{n-1} A,X).
\end{equation}

Since $(k-1)$-$\dell B\leq n$, $\Omega^n B\dirsum \Omega^{n+k-1}N_B$ for some module $N_B$, so 
\[
\Ext_{\Lambda}^{k'}(\Omega^n B,X) \dirsum \Ext_{\Lambda}^{k'}(\Omega^{n+k-1}N_B,X) \cong \Ext_{\Lambda}^{n+k-1+k'}(N_B,X)=0.
\]

Since $\kdell A\leq n-1$, $\Omega^{n-1} A\dirsum \Omega^{n+k-1}N_A$ for some module $N_A$, so
\[
\Ext_{\Lambda}^{k'}(\Omega^{n-1} A,X)\dirsum \Ext_{\Lambda}^{k'}(\Omega^{n+k-1} N_A,X) \cong \Ext_{\Lambda}^{n+k-1+k'}(N_A,X)=0.
\]

This makes the first and last term of the exact sequence \eqref{eq:Ext_sequence} both 0, so $\Ext_{\Lambda}^{k'}(\Omega^n C,X)\cong \Ext_{\Lambda}^{n+k'}(C,X)=0$. This shows the $(k-1)$-effective delooping level of $C$ is at most $n$.
\end{proof}

{\color{black}
\begin{proof}[Proof of Theorem \ref{thm:ddell}]
Let $\kddell\Lambda=m<\infty$. For any simple $S$ and exact sequence
\begin{equation}
\label{eq:ddell_sequence}
0 \xrightarrow{f_{n+1}} C_n \xrightarrow{f_n} C_{n-1} \to \cdots \to C_1 \xrightarrow{f_1} C_0 \xrightarrow{f_0} S \to 0
\end{equation}
such that $(i+k)$-$\dell C_i\leq m-i$ for $i=1,2,\dots,n$, we can split \eqref{eq:ddell_sequence} into $n$ short exact sequences:
\[
0\to \coker f_{i+1} \xrightarrow{f_i} C_{i-1} \xrightarrow{f_{i-1}} \coker f_i \to 0,
\]
where $i=1,2,\dots,n$, and in particular $\coker f_{n+1}=C_n$, $\coker f_1=S$.

For $i=n$, we have the short exact sequence $0\to C_n \xrightarrow{f_n} C_{n-1} \xrightarrow{f_{n-1}} \coker f_n \to 0$ satisfying the conditions of Lemma \ref{lem:v1_induction}, so $(n+k-1)$-$\edell(\coker f_n) \leq m-n+1$.

For $i=n-1$, the short exact sequence $0\to \coker f_n \xrightarrow{f_{n-1}} C_{n-2} \xrightarrow{f_{n-2}} \coker f_{n-1} \to 0$ satisfies the conditions of Lemma \ref{lem:v2_induction}, so $(n+k-2)$-$\edell(\coker f_{n-1})\leq m-n+2$. Inductively, we apply Lemma \ref{lem:v2_induction} to $i=n-2, n-3, \dots,1$ to conclude that $k$-$\edell S\leq m$. Therefore, the theorem follows when $k=1$.
\end{proof}
}

We obtain Theorem \ref{thm:theorem1} by combining Proposition \ref{prop:edell_vs_findim}, Theorem \ref{thm:subddell}, and Theorem \ref{thm:ddell} when $k=1$.

{\color{black}
\section{Derived delooping level of all modules}
\label{sec:trivial_extensions}

In \cite{kershaw2023}, the authors use the construction in \cite{ringel2020gorenstein} to give an example of an algebra with infinite delooping level. Here we will show, as a consequence of very general lemmas, that the example of \cite{kershaw2023} has finite $\kddell\!$ for all $k\ge1$.

\begin{comment}
The first comment is that we only need the $\kddell$ of simple modules.
\end{comment}

\begin{lemma}\label{lemma: kddell finite is closed under extension}
Suppose that $0\to A\to B\to C\to 0$ is a short exact sequence with $\kddell A=m_1$, $\kddell C=m_2$. Then $\kddell B\leq m_1+m_2+1$.
\end{lemma}

\begin{proof}
By assumption, we have two exact sequences
\begin{equation}
\label{eq:exact sequence for A}
0\to D_{n_1} \to D_{n_1-1} \to \cdots \to D_0 \to A\to 0,
\end{equation}
\begin{equation}
\label{eq:exact sequence for C}
0\to E_{n_2} \to E_{n_2-1} \to \cdots \to E_0 \to C\to 0,
\end{equation}
where $(i+k)$-$\dell D_i\leq m_1-i$ for $i=0,\dots,n_1$ and $(i+k)$-$\dell E_i\leq m_2-i$ for $i=0,\dots,n_2$.

We also need the first $n_1+1$ steps of the minimal projective resolution of $C$
\begin{equation}
\label{eq:proj_res_of_C}
0\to \Omega^{n_1+1} C \to P_{n_1} \to \cdots \to P_1 \to P_0 \to C \to 0,
\end{equation}
and another exact sequence
\begin{equation}
\label{eq:syzygy_exact_seq_for_C}
0\to \Omega^{n_1+1} E_{n_2} \to \Omega^{n_1+1} E_{n_2-1} \to \cdots \to \Omega^{n_1+1} E_0 \to \Omega^{n_1+1} C \to 0
\end{equation}
by applying Lemma \ref{lem:rotating_ses} to \eqref{eq:exact sequence for C}.

Combining \eqref{eq:exact sequence for A}, \eqref{eq:proj_res_of_C}, and \eqref{eq:syzygy_exact_seq_for_C}, it is straightforward to check that we get the long exact sequence
\begin{equation}
\label{eq:long_exact_seq_lemma}
0\to \Omega^{n_1+1}E_{n_2} \to \cdots \to \Omega^{n_1+1}E_0 \overset{f}{\to} D_{n_1}\oplus P_{n_1} \to \cdots \to D_1\oplus P_1 \to D_0\oplus P_0 \overset{g}\to B \to 0,
\end{equation}
where $f$ maps into $P_{n_1}$, which can be factored as $\Omega^{n_1+1}E_0\onto \Omega^{n_1+1}C \inc P_{n_1}$, and $g$ can be factored as $D_0\oplus P_0\onto A\oplus P_0\onto B$ and the last surjective map $P_0 \onto B$ is $g'$ in Lemma \ref{lem:rotating_ses}.

Now, for $\kddell B\leq m_1+m_2+1$, it remains to check
\begin{enumerate}
\item $(i+k)$-$\dell D_i \leq m_1+m_2+1-i$ for $i=0,\dots,n_1$,
\item $(i+n_1+1+k)$-$\dell\Omega^{n_1+1+i} E_i\leq m_1+m_2-i-n_1$ for $i=0,\cdots,n_2$.
\end{enumerate}

A stronger version of the first statement holds due to the conditions in \eqref{eq:exact sequence for A}. For the second statement, it is equivalent to show $\Omega^{m_1+m_2+1} E_i\dirsum \Omega^{m_1+m_2+1+k} N_i$ for some $N_i$ and $i=0,\dots,n_2$, but this is immediately implied from the conditions in \eqref{eq:exact sequence for C}.
\end{proof}

The previous lemma implies that if $\kddell\Lambda<\infty$ for some $k$, then $\kddell M<\infty$ for all finitely generated $\Lambda$-modules $M$, while the same statement is not true for the delooping level by Remark \ref{rem: dell not closed under ext} and \cite{kershaw2023}.

\begin{lemma}\label{lemma: ddell of syzygies}
If $M$ has bounded $\kddell\!$, then $\Omega M$ has bounded $(k+1)\text-\ddell\!$ and therefore bounded $\kddell\!$.
\end{lemma}

\begin{proof}
Suppose $\kddell M=m$ using $n$ is finite. Assume $m\neq n$ and $m\neq 0$. Then there is an exact sequence
\begin{equation}
\label{eq:lem:ddell of syzygies}
0\to C_n \to C_{n-1} \to \cdots \to C_0 \to M \to 0,
\end{equation}
where $(i+k)$-$\dell C_i\leq m-i$ for $i=0,\dots,n$, or equivalently, $(i+k+1)$-$\dell \Omega C_i\leq m-i-1$ for $i=0,\dots,n$.

Applying Lemma \ref{lem:rotating_ses}, we obtain another exact sequence
\[
0\to \Omega C_n \to \cdots \to \Omega C_0 \to \Omega M \to 0,
\]
which implies $(k+1)$-$\ddell \Omega M \leq m-1$.

If $m=n$, instead of having $(i+k+1)$-$\dell \Omega C_i\leq m-i-1$, we have $(i+k+1)$-$\dell \Omega C_i\leq m-i$ in \eqref{eq:lem:ddell of syzygies}, so $(k+1)$-$\ddell \Omega M \leq m$.

If $m=0$, $\kddell M = \kdell M=0$, so $(k+1)$-$\dell \Omega M = (k+1)$-$\ddell \Omega M = 0$.
\end{proof}

\begin{theorem}
    The $\kddell\!$ of any local algebra $\Lambda$ is finite for all $k$.
\end{theorem}

\begin{proof}
Local algebras only have one simple module $S$, so $\dell S = \ddell S = 0$. By Lemma \ref{lemma: kddell finite is closed under extension}, $\ddell M$ is finite for all finitely generated modules $M$. Take $M=\Lambda/S$. Then $S=\Omega M$, so $2$-$\ddell S$ is finite by Lemma \ref{lemma: ddell of syzygies}. This implies $2$-$\ddell M$ is finite for all $M$. Proceeding like so, we get $\kddell M<\infty$ for all $\Lambda$-modules $M$.
\end{proof}

\begin{corollary}
Let $M$ be any (finitely generated) module over a local algebra $A$ and consider the one-point extension $\Lambda=A[M]$. Then $\kddell \Lambda$ is finite for all $k$.
\end{corollary}

\begin{proof}
    Let $S$ be the unique simple $A$-module and let $T$ be the other simple $\Lambda$-module. We know that $\kddell S<\infty$ for all $k$. Also, $(k+1)$-$\ddell M<\infty$ for all $k$. This implies that $\kddell T<\infty$ for all $k$.
\end{proof}

More generally, this same proof shows that, for $\Lambda=A[M]$ to have finite $\kddell\!$, it suffices for $A$ to have finite $\kddell\!$ and for $M$ to have finite $(k+1)$-$\ddell\!$.

\begin{remark}
The example of \cite{kershaw2023} is an algebra with infinite delooping level but finite derived delooping level. This shows that the derived delooping level is a stronger invariant than the original delooping level. In fact, we will show next that $\ddell\Lambda = \Findim\Lambda^{\op}=1$ in their case.
\end{remark}
}

\begin{example}[Example of $\Lambda$ where $\dell\Lambda=\infty$ in \cite{kershaw2023} and \cite{ringel2020gorenstein}. Also, $\subddell\Lambda=\infty$, but $\ddell\Lambda=\Findim\Lambda^{\op}=1$]
\label{ex:infinite_delooping_level}

We first summarize the results from \cite{ringel2020gorenstein} and \cite{kershaw2023} in our context. Let $K$ be a field and $q\in K$ have infinite multiplicative order. Let $A$ be the six-dimensional algebra $K\langle x,y,z\rangle/I$, where $I=(x^2, y^2, z^2, zy, yx+qxy, zx-xz, yz-xz)$. Then the quiver of $A$ is $\begin{tikzcd} 1 \arrow[loop right,"y"] \arrow[loop below, "x"] \arrow[loop above, "z"] \end{tikzcd}$ subject to the relations in $I$. The indecomposable projective $A$-module is six-dimensional and can be visualized as

\begin{equation}
\label{mod:local_algebra}
\begin{tikzcd}
{} & {} & e_1 \arrow[d, dash, "y"] \arrow[dll, dash, "x", swap] \arrow[drr, dash, "z"] & {} & {} \\
x \arrow[dr, dash, "-qy", swap] \arrow[drrr, dash, "z" near start] & {} & y \arrow[dl, dash, "x" near start, swap] \arrow[dr, dash, "z"] & {} & z \arrow[dl, dash, "x"] \\
{} & yx=-qxy & {} & xz=yz=zx & {}
\end{tikzcd}
\end{equation}
\end{example}

\textbf{Notation.} For easier reference, we will let $X=yx=-qxy$ and $Y=xz=yz=zx$.

For all $\alpha\in K$, define $M(\alpha)$ as the three-dimensional $A$-module with basis $v, v', v''$ such that $vx=\alpha v'$, $vy=v'$, $vz=v''$. Let $\Lambda = A[M(q)]$ be the one-point extension. The quiver of $\Lambda$ is $\begin{tikzcd} 2 \arrow[r, "p"] & 1 \arrow[loop right,"y"] \arrow[loop below, "x"] \arrow[loop above, "z"] \end{tikzcd}$. The indecomposable projective $\Lambda$-module at 2 is
$P_2 = \begin{matrix} 2 \\ M(q) \end{matrix} = 
\begin{tikzcd} {} & 2 \arrow[d, dash, "p"] & {} \\
{} & v \arrow[dl, dash, "x", shift right, swap] \arrow[dl, dashed, no head, "qy", shift left] \arrow[dr, dash, "z"] & {} \\
v' & {} & v'' \end{tikzcd}$,
where the dashed segment $qy$ is used to indicate $v\cdot x=qv'=v\cdot qy$.

The indecomposable projective $\Lambda$-module $P_1$ at 1 is isomorphic to \eqref{mod:local_algebra}. We know from \cite{kershaw2023} that $\dell \Lambda =\dell\!_{\Lambda} S_2 =\infty$, while $\Findim\Lambda^{\op}=1$. We also know that $\subddell\Lambda=\dell\!_{\Lambda} S_2=\infty$ since the simple module at 2 does not embed in any other module other than itself.

However, we will prove $\Lambda=A[M(q)]$ has $\ddell\!$ equal to $\Findim\Lambda^{\op}=1$. We first show that the simple module at 1 is infinitely deloopable. Note that $\Omega M(1)$ has $S_1$ as a direct summand because of the short exact sequence

\begin{equation}
\label{eq:syzygy_of_M1}
0 \to \begin{tikzcd} x-y \arrow[d, dash, "x", shift right, swap] \arrow[d, dashed, no head, "qy", shift left] \\ X \end{tikzcd} \oplus Y
\overset{g}{\to}
\begin{tikzcd}[column sep=tiny]
{} & {} & e_1 \arrow[d, dash, "y"] \arrow[dll, dash, "x", swap] \arrow[drr, dash, "z"] & {} & {} \\
x \arrow[dr, dash, "-qy", swap] \arrow[drrr, dash, "z" near start] & {} & y \arrow[dl, dash, "x" near start, swap] \arrow[dr, dash, "z"] & {} & z \arrow[dl, dash, "x"] \\
{} & X & {} & Y & {} \end{tikzcd}
\overset{f}{\to}
\begin{tikzcd} {} & v \arrow[dl, dash, "x", shift right, swap] \arrow[dl, dashed, no head, "y", shift left] \arrow[dr, dash, "z"] & {} \\
v' & {} & v'' \end{tikzcd}
\to 0,
\end{equation}
where $f$ is the canonical projection sending $e_1$ to $v$ and the kernel of $g$ is written in a way that it embeds naturally into $P_1$.

The module $M(1)$ is special in that $(x-y)z=0$, so the three dimensional $\Omega M(1)$ is decomposable. On the other hand, if $\alpha\neq 1$, it can be shown that $\Omega M(\alpha) = M(q\alpha)$ is indecomposable \cite[Lemma 6.4]{ringel2020gorenstein}. Therefore, we conclude that $S_1$ is infinitely deloopable since
\begin{equation}
\label{eq:1_infinitely_deloopable}
S_1 \dirsum \Omega^{n+1} M(q^{-n})
\end{equation}
for all $n\in\Z_{>0}$.

Consider the short exact sequence in $\mod\Lambda$
\begin{equation}
\label{eq:local_example}
0 \to S_1 \to \begin{matrix} 2 \\ v \end{matrix} \to S_2 \to 0.
\end{equation}

Since $\Omega \left( \begin{matrix} 2 \\ v \end{matrix} \right) = S_1 \oplus S_1$ and $S_1$ is infinitely deloopable, $\dell\!_{\Lambda} \!\left( \begin{matrix} 2 \\ v \end{matrix} \right) = 1$ and 2-$\dell\!_{\Lambda} S_1 = 0$. By Definition \ref{def:ddell}, $\ddell\!_{\Lambda} S_2 = 1$, and therefore $\ddell \Lambda = 1$.

\begin{remark}
\label{rem: dell not closed under ext}
By \cite[Proposition 4.1]{kershaw2023}, $\dell\!_A M(q)=\infty$ and $\dell\!_A S_1 = 0$. Since $M(q)$ is an iterated extension of $S_1$, this shows that the set of modules with finite delooping level is not closed under extensions. On the other hand, since $M(q)$ is not a submodule of $A$, $\ddell\!_A M(q) \geq 1$. By Theorem \ref{lemma: kddell finite is closed under extension}, $\ddell\!_A M(q) \leq 1$ because $M(q)$ is the extension of $S_1\oplus S_1$ and $S_1$, where $\dell\!_A S_1 = 0$. Therefore, $\ddell\!_A M(q) = 1$.
\end{remark}

\section{Sufficient condition for $\findim\Lambda<\infty$}
\label{sec:dell_vs_phidim}

In this section, we provide a sufficient condition for the finiteness of findim by comparing $\dell\Lambda$ and $\phidim\Lambda$. Note that this comparison is done on the same algebra $\Lambda$ instead of opposite algebras. Both invariants focus on the structure of syzygies of simple $\Lambda$-modules, so it is not surprising that they may be related.

{\color{red}
%We repeat the setup in \cite{igusa2005} in order to recall the definition of the $\phi$-dimension. Let $K_0$ be the free abelian group generated by all symbols $[M]$, where $M$ is an indecomposable $\Lambda$-module, modulo the subgroup generated by $[P]$ for projective modules $P$ and by $[M]-[N]-[N']$ for $M\cong N \oplus N'$. 
}

We repeat the setup in \cite{igusa2005} in order to recall the definition of the $\phi$-dimension. Let $K_0$ be the free abelian group generated by all symbols $[M]$, where $M$ is a finitely generated $\Lambda$-module, modulo the subgroup generated by $[P]$ for projective modules $P$ and by $[M]-[N]-[N']$ for $M\cong N \oplus N'$. Thus $K_0$ has a free basis given by those $[M]$ where $M$ is indecomposable. The elements of $K_0$ are $[M]-[N]$ where $M,N$ are $\Lambda$-modules.

Define the endomorphism
\[
L: K_0\to K_0 \text{ by } [M]\mapsto [\Omega M].
\]

Let $\add M$ be the additive category of the module $M$ and let $\langle\add M\rangle$ denote the subgroup of $K_0$ generated by $[N]$ for all $N\in \add M$. More generally, for any finite set of modules $T$, we denote by $\langle \add T\rangle$ the subgroup of $K_0$ generated by $[M]$ for all $M\in \add T$. Then, we define
\[
\phi(M) = \inf\{n \mid L^m(\langle\add M\rangle) \cong L^{m+1}(\langle\add M\rangle) \, \forall m\geq n\},
\]
\[
\phidim\Lambda = \sup\{\phi (M) \mid M\in \mod\Lambda\}.
\]

We extend the definition slightly for when the additive category is generated by some subset $T$ of $\mod\Lambda$, that is,
\begin{align*}
\phidimT\Lambda & = \inf\{n \mid L^m(\langle \add T\rangle) \cong L^{m+1}(\langle \add T\rangle) \, \forall m\geq n\} %\\
%& = \sup\{\phi (M) \mid M\in T\}. (not true)
\end{align*}

%Since $\phi(M)$ is finite for all finitely generated modules $M$ 
By Fitting's Lemma, $\phidimT\Lambda$ is always finite if $T$ is a finite set.

We could do the same for the delooping level and define $\dellT\Lambda := \sup \{ \dell M \mid M\in T\}$. However, note that when $T$ contains all simple modules and is contained in all indecomposable summands of syzygies of simple modules, as in the case of Theorem \ref{thm:dell_phidim_comp}, $\dellT\Lambda = \dell\Lambda$, so this definition is not useful.%unimportant.

\begin{theorem}
\label{thm:dell_phidim_comp}
For any finite dimensional algebra $\Lambda$ over a field $K$, let $T_{\Lambda}$ be the set of non-projective indecomposable summands of syzygies of simple $\Lambda$-modules, including all simple modules. If $T_{\Lambda}$ is a finite set, then
\[
\Findim\Lambda^{\op} \leq \subddell\Lambda \text{ or } \ddell\Lambda \leq \dell \Lambda \leq \phidimT\Lambda \leq \phidim\Lambda,
\]
and in particular, since $\phidimT\Lambda$ is finite, the finitistic dimension conjecture holds for $\Lambda^{\op}$.
\end{theorem}

\begin{proof}
Let $T=\{T_1,T_2,\dots,T_N\}$. Note that we can restrict $L$ to be an endomorphism of $\langle \add T\rangle$ since $T$ is closed under taking syzygies. 
Thus, $L\langle \add T\rangle\subset \langle \add T\rangle$ which implies $L^{n+1}\langle \add T\rangle\subset L^n\langle \add T\rangle$. It is clear that $\phidimT \Lambda$ is finite since $T$ is finite, so let $\phidimT\Lambda=n$. Then $L^n\langle \add T\rangle \cong L^{n+1}\langle \add T\rangle$ as free abelian groups of finite rank. This makes the quotient $L^n\langle \add T\rangle / L^{n+1}\langle \add T\rangle$ into a finite abelian group of order, say $m$. Then $mL^n\langle \add T\rangle \subset L^{n+1}\langle \add T\rangle$.

This implies that, for all $T_i\in T$, $m[\Omega^n T_i]=[\Omega^{n+1}A]-[\Omega^{n+1}B]$ for some $A,B\in \add T$. So, $\Omega^nT_i$ is a direct summand of $\Omega^{n+1}A$ which makes $\dell T_i\le n=\phidimT\Lambda$ for all $i$.
%Let $l$ be their common rank. If $l=0$, then $\phidimT\Lambda = n$ and $\dellT\Lambda\leq n$. If $l\neq 0$, by reordering the elements in $T$, we can pick $T_1,\dots,T_l$ in $T$ so that $\{[\Omega^n T_1],\dots, [\Omega^n T_l]\}$ is a basis for $L^n\langle \add T\rangle$. In particular, every element of the form $[\Omega^n M]$ in $L^n\langle \add T\rangle$ is a unique non-negative linear combination of $[\Omega^n T_i]$.
%Since $L:L^n\langle \add T\rangle \to L^{n+1}\langle \add T\rangle$ is an isomorphism, $\{[\Omega^{n+1} T_1],\dots, [\Omega^{n+1} T_l]\}$ is basis for $L^{n+1}\langle \add T\rangle$ and must be linearly independent in $L^n\langle \add T\rangle$. Now we can write for each $i\in\{1,\dots,l\}$ that
%\[\Omega^{n+1} T_i = \Omega^n(\Omega T_i) = \Omega^n \left( \bigoplus_{j} T_{j}^{k_{ij}} \right) = \bigoplus_j \Omega  T_{j}^{k_{ij}}\]
%\begin{equation}
%\label{eq:lincomb}
%\textit{i.e.} \quad [\Omega^{n+1} T_i] = \sum_j k_{ij} [\Omega^n T_{j}],
%\end{equation}
%where $k_{ij}\in\Z_{\geq 0}$ and for a fixed $i$, the collection of $k_{ij}$ is not all zero. Moreover, for each $j\in\{1,\dots,l\}$, there exists $i\in\{1,\dots,l\}$ such that $[\Omega^n T_j]$ occurs as a summand in $[\Omega^{n+1}T_i]$ as in \eqref{eq:lincomb}, that is, $\Omega^n T_j$ is a direct summand of $\Omega^{n+1} T_i$, hence a stable retract of $\Omega^{n+1} T_i$. The reason is that if not, then $[\Omega^{n+1}T_1],\dots,[\Omega^{n+1}T_l]$ are not linearly independent.
%Therefore, $\dell T_i\leq n=\phidimT\Lambda$ for all $i$, 
So $\dell \Lambda\leq \phidimT \Lambda<\infty$. Since $\Findim\Lambda^{\op} \leq \dell\Lambda$, $\Findim\Lambda^{\op}<\infty$.
\end{proof}

We would like to point out that there is a similar concept in Example 1.22 of \cite{gelinas2022} called ``$n$-syzygy finiteness.'' It is straightforward to show that if $\Lambda$ is $n$-syzygy finite, then $\phidim\Lambda<\infty$, but the converse might not be true. Investigations on their relationships will be the topic of a future paper. {\color{black} Another similar concept is ``finite cosyzygy type'' which appears in Definition 7.1 of \cite{rickard2019unbounded}, where Rickard shows in Lemma 6.1 and Proposition 7.2 that if all simple $\Lambda$-modules have finite cosyzygy type, then $\Findim\Lambda<\infty$. This is equivalent to our previous Theorem \ref{thm:dell_phidim_comp}. Our theorem has the slight improvement that in addition to knowing $\Findim\Lambda$ is finite, we can find an upper bound $\Findim\Lambda\leq \phidimT \Lambda^{\op}$ if $T$ is finite.}

Theorem \ref{thm:dell_phidim_comp} immediately recovers the result that $\findim\Lambda<\infty$ if $\Lambda$ is monomial.

\begin{corollary}[\cite{green1991}]
\label{cor:findim_monomial}
The findim conjecture holds for monomial algebras.
\end{corollary}

\begin{proof}
We write the monomial algebra $\Lambda$ as $kQ/I$. By \cite{huisgen1991predicting}, the second syzygy of any module is of the form $\oplus p\Lambda$ where $p$ is a path in $Q$ of length $\geq 1$, so $T_{\Lambda}$ is finite.
\end{proof}

The one-point extension algebra $\Lambda$ in Example \ref{ex:infinite_delooping_level} is one where $T_{\Lambda}$ is infinite because $\Omega^n S_1$ gets wider (\textit{i.e.} has more direct summands in its top) as $n$ increases. Another example is the counterexample to the $\phi$-dimension conjecture presented in \cite{hanson2022counterexample}. We present an easy-to-verify example where the set $T_{\Lambda}$ is infinite where $\Lambda$ is special biserial.

\begin{example}
\label{ex:wider_syzygies}
Consider the following quiver with 9 vertices on a cylinder. Let $x, y$ mean left and right arrows, respectively.

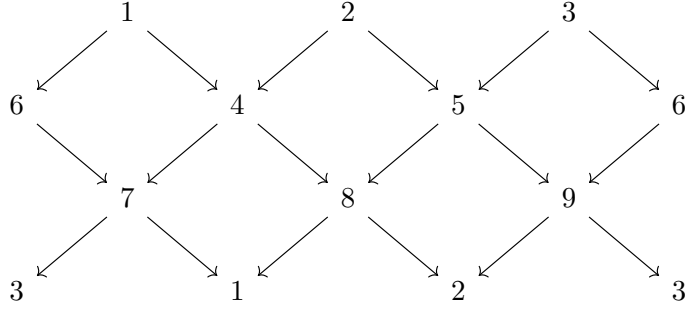
\begin{center}
\begin{tikzcd}
{} & 1 \arrow[dl] \arrow[dr] & {} & 2 \arrow[dl] \arrow[dr] & {} & 3 \arrow[dl] \arrow[dr] & {} \\
6 \arrow[dr] & {} & 4 \arrow[dl] \arrow[dr] & {} & 5 \arrow[dl] \arrow[dr] & {} & 6 \arrow[dl] \\
{} & 7 \arrow[dl] \arrow[dr] & {} & 8 \arrow[dl] \arrow[dr] & {} & 9 \arrow[dl] \arrow[dr]& {} \\
3 & {} & 1 & {} & 2 & {} & 3
\end{tikzcd}
\end{center}

Let $\Lambda=KQ/(xy-yx, x^2, y^2)$. Then the syzygy of any simple module gets wider. For example,
\[
\Omega S_2 = \begin{matrix} 4 & {} & 5 \\ {} & 8 & {} \end{matrix}, \quad \Omega^2 S_2 = \begin{matrix} 7 & {} & 8 & {} & 9 \\ {} & 1 & {} & 2 & {} \end{matrix}, \quad \Omega^3 S_2 = \begin{matrix}
3 & {} & 1 & {} & 2 & {} & 3 \\
{} & 6 & {} & 4 & {} & 5
\end{matrix}, \text{and so on.}
\]
\end{example}

\section{Future Directions}
\label{sec:future_directions}
The introduction of new invariants brings forth more questions. First of all, we can still explore more of their relationships among themselves. For example,
\begin{question}
Regarding the two new sub-derived and derived delooping levels,
\begin{enumerate}
\item Can we compare $\ksubddell\Lambda$ and $\kddell\Lambda$? We saw in Section \ref{sec:trivial_extensions} that $\ddell\Lambda<\subddell\Lambda$ is possible. Does the inequality $\ddell\Lambda\leq \subddell\Lambda$ hold for all Artin algebras $\Lambda$ and for general $k$? To what extent can we compare $\ddell\Lambda$ and $\subddell\Lambda$?
\item Is $\ddell\Lambda = \Findim\Lambda^{\op}$ true? There is no example that we know of where this is false.
\item If the answer to the previous question is negative, can we quantify the difference $\ddell\Lambda-\Findim\Lambda^{\op}$?
\item Can we use these new invariants to prove the findim conjecture in other settings?
\end{enumerate}
\end{question}

Moreover, we would like to point out that there is an associated torsion pair.

\begin{theorem}\label{thm: finite ddell is torsion-free class}
  The class of modules $\mathcal{F}$ with finite derived delooping level forms a torsion-free class. If this class contains all $\Lambda$-modules, then $\Findim \Lambda^{\op}<\infty$.  
\end{theorem}

\begin{proof}
    Lemma \ref{lemma: kddell finite is closed under extension} shows $\mathcal{F}$ is closed under extensions. The next lemma shows that $\mathcal{F}$ is closed under submodules. Thus $\mathcal F$ is a torsion-free class. If $\mathcal F$ contains all modules, it contains the simple modules making $\ddell\Lambda<\infty$. Then, $\Findim \Lambda^{\op}<\infty$.
\end{proof} 

\begin{remark}
    Note that $\mathcal F$ contains all $\Lambda$-modules if and only if the corresponding torsion class $\mathcal G=\,^\perp\mathcal F$ is zero.
\end{remark}

\begin{lemma}
\label{lem:kddell_closed_under_submods}
The set $\mathcal{F}$ is closed under submodules.
\end{lemma}
\begin{proof}
If $N\in\mathcal{F}$ and $M$ is a submodule of $N$, then $0\to M \overset{f}{\to} N \to \coker f\to 0$ rotates to $0\to \Omega\coker f\to M \to N\to 0$ by Lemma \ref{lem:rotating_ses}, where $\ddell N<\infty$ and $\ddell(\Omega\coker f)\leq \dell(\Omega\coker f)=0$. Therefore, $M\in\mathcal{F}$ by Lemma \ref{lemma: kddell finite is closed under extension}.
\end{proof}

If $\ddell\Lambda$ is finite, then all simple $\Lambda$-modules belong to $\mathcal{F}$, so the torsion-free class $\mathcal{F}$ contains all finitely generated $\Lambda$-modules. The corresponding torsion class $\mathcal{T}$ would be empty. In general, we could ask if $\mathcal T$ is finitely generated. In $\Lambda^{\op}$, we will have the torsion pair $(D\mathcal{F},D\mathcal{T})$. This connection to tilting theory allows us to potentially answer questions about findim using the derived delooping level for the same or opposite algebra.

Since $\mathcal F$ contains all finitely generated $\Lambda$-modules if and only if $D\mathcal F$ contains all finitely generated $\Lambda^{\op}$-modules and $D\mathcal F$ contains all finitely generated injective $\Lambda^{\op}$-modules, we have the following corollary of Theorem \ref{thm: finite ddell is torsion-free class}.

\begin{corollary}
    If the torsion class generated by all finitely generated injective $\Lambda$-modules contains all finitely generated $\Lambda$-modules, then $\Findim\Lambda<\infty$.
\end{corollary}

\begin{proof} Let $\mathcal J$ be the torsion class generated by all injective $\Lambda$-modules. If $\mathcal J=\mod\text-\Lambda$, then the torsion-free class $D\mathcal{J}$ in $\mod\Lambda^{\op}$ cogenerated by all projective $\Lambda^{\op}$-modules contains all $\Lambda^{\op}$-modules and in particular all simple $\Lambda^{\op}$-modules. Since $\mathcal F$ contains all projective $\Lambda^\op$-modules, we have $D\mathcal J\subseteq \mathcal F$ and $D\mathcal{J} = \mod\text-\Lambda^\op$. So, by Theorem \ref{thm: finite ddell is torsion-free class}, $\Findim(\Lambda^{\op})^{\op}=\Findim\Lambda<\infty$.
\end{proof}

This leads us to a comparison of derived delooping level and injective generation of the derived category. Let $\D(\Lambda)$ be the unbounded \textbf{derived category} of complexes over $\Lambda$-modules. Let $M^{\perp}$ be the \textbf{right perpendicular category} of $M$, defined as the full subcategory of $\D(\Lambda)$ such that $\Hom_{\D(\Lambda)}(M, X[t])=0$ for all $X$ and $t\in\Z$. The left perpendicular category $\prescript{\perp}{}{M}$ of $M$ is defined similarly. In \cite[Theorem 4.3]{rickard2019unbounded}, Rickard proves that if the localizing subcategory generated by all injective $\Lambda$-modules (the smallest triangulated subcategory of $\D(\Lambda)$ that contains all injectives and is closed under coproducts), denoted by $\langle \text{Inj-}\Lambda \rangle$, is the entire $\D(\Lambda)$, then $\Findim\Lambda<\infty$. In that case, we say that \textbf{injectives generate} for $\Lambda$. In particular, if all simple stalk complexes are in $\langle \text{Inj-}\Lambda \rangle$, then injectives generate. The paper presents several methods for showing injective generation, but it is difficult in general to determine whether a (simple) stalk complex is in $\langle \text{Inj-}\Lambda \rangle$. In light of our new invariant, we present two possible candidates for $\langle \text{Inj-}\Lambda \rangle$ in terms of $\ddell\!$.

Suppose $\ddell\Lambda^{\op}=\infty$ and $\ddell {DS}=\infty$ for some simple $\Lambda$-module $DS_{\Lambda}$. Then we would like to ask whether the dual simple module $S$ is in the localizing subcategory generated by all injective $\Lambda$-modules $I$. Since we know $\prescript{\perp}{}{S}$ is a localizing subcategory, the question becomes whether $\Hom_{\D(\Lambda)}(I, S[t])=\Ext_{\Lambda}^t(I, S)=0$ for all $t\in \N$. It is clear that this is true for $t=0$, but there does not seem to be an easy way to determine the case when $k>0$.

Define a class of $\Lambda$-modules $\F_{\infty}=\{DM \mid \kddell M<\infty \text{ for all $k\in\N$}\}$. Another candidate for $\langle \text{Inj-}\Lambda \rangle$ is $\langle \F_{\infty} \rangle$. Note that all projective $\Lambda^{\op}$-modules have finite $\kddell\!$ for all $k$, so $\F_{\infty}$ contains all injective $\Lambda$-modules. Moreover, $\F_{\infty}$ satisfies the following.
\begin{lemma}
\label{lem:F infinity is dense subcat}
If $0\to A\to B\to C\to 0$ is a short exact sequence in $\D(\Lambda)$, then if two of the terms are in $\F_{\infty}$ then so is the third, where $A, B, C$ are all considered as stalk complexes.
\end{lemma}

\begin{proof}
If $B,C$ are in $\F_{\infty}$, then so is $A$ by Theorem \ref{thm:subddell}. If $A,B$ are in $\F_{\infty}$, then so is $C$ by Theorem \ref{thm:ddell}. Finally, if $A,C\in \F_{\infty}$, then so is $B$ by Lemma \ref{lemma: kddell finite is closed under extension}. 
\end{proof}

We summarize the two candidates as two questions below.

\begin{question}
\label{question:ddell and injective generation}
Suppose there is a simple $\Lambda$-module $S$ such that $\ddell DS=\infty$ in $\Lambda^{\op}$.
\begin{enumerate}
\item Is it true that $\Ext_{\Lambda}^t(I,S)=0$ for all injective modules $I$ and all $t\in\N$?
\item Is the stalk complex $S$ not contained in the localizing subcategory $\langle \F_{\infty} \rangle$?
\end{enumerate}
\end{question}

A positive answer to either question would suggest that having finite derived delooping level is a stronger condition than injective generation. Also, $\ddell \Lambda^{\op}$ gives a good upper bound for $\Findim \Lambda$. Investigating these questions will also tighten our understanding of various prominent methods for solving the findim conjecture.

\bibliographystyle{plain}
\bibliography{refs}

\begin{thebibliography}{10}

\bibitem{auslander1975generalized}
Maurice Auslander and Idun Reiten.
\newblock On a generalized version of the nakayama conjecture.
\newblock {\em Proceedings of the American Mathematical Society}, 52(1):69--74,
  1975.

\bibitem{bass1960}
Hyman Bass.
\newblock Finitistic dimension and a homological generalization of semi-primary
  rings.
\newblock {\em Transactions of the American Mathematical Society},
  95(3):466--488, 1960.

\bibitem{erdmann2004radical}
Karin Erdmann, Thorsten Holm, Osamu Iyama, and Jan Schr{\"o}er.
\newblock Radical embeddings and representation dimension.
\newblock {\em Advances in mathematics}, 185(1):159--177, 2004.

\bibitem{gelinas2022}
Vincent G{\'e}linas.
\newblock The depth, the delooping level and the finitistic dimension.
\newblock {\em Advances in Mathematics}, 394:108052, 2022.

\bibitem{goodearl2014repetitive}
KR~Goodearl and B~Huisgen-Zimmermann.
\newblock Repetitive resolutions over classical orders and finite dimensional
  algebras.
\newblock {\em arXiv preprint arXiv:1407.2321}, 2014.

\bibitem{GKK1991}
Edward~L Green, Ellen Kirkman, and James Kuzmanovich.
\newblock Finitistic dimensions of finite dimensional monomial algebras.
\newblock {\em Journal of algebra}, 136(1):37--50, 1991.

\bibitem{green1991}
Edward~L Green and Birge Zimmermann-Huisgen.
\newblock Finitistic dimension of artinian rings with vanishing radical cube.
\newblock {\em Mathematische Zeitschrift}, 206:505--526, 1991.

\bibitem{masters_thesis}
Jacob~Fjeld Grevstad.
\newblock Finitistic dimension conjecture, 2021.

\bibitem{hanson2022counterexample}
Eric~J Hanson and Kiyoshi Igusa.
\newblock A counterexample to the $\phi$-dimension conjecture.
\newblock {\em Mathematische Zeitschrift}, 300(1):807--826, 2022.

\bibitem{happel1991homological}
Dieter Happel.
\newblock Homological conjectures in representation theory of
  finite-dimensional algebras.
\newblock {\em Sherbrook Lecture Notes Series}, 1991.

\bibitem{happel1993reduction}
Dieter Happel.
\newblock Reduction techniques for homological conjectures.
\newblock {\em Tsukuba journal of mathematics}, 17(1):115--130, 1993.

\bibitem{huisgen1991predicting}
Birge~Zimmermann Huisgen.
\newblock Predicting syzygies over monomial relations algebras.
\newblock {\em manuscripta mathematica}, 70:157--182, 1991.

\bibitem{huisgen1992homological}
Birge~Zimmermann Huisgen.
\newblock Homological domino effects and the first finitistic dimension
  conjecture.
\newblock {\em Inventiones mathematicae}, 108(1):369--383, 1992.

\bibitem{huisgen1995}
Birge~Zimmermann Huisgen.
\newblock The finitistic dimension conjectures--a tale of 3.5 decades.
\newblock In {\em Abelian Groups and Modules: Proceedings of the Padova
  Conference, Padova, Italy, June 23--July 1, 1994}, pages 501--517. Springer,
  1995.

\bibitem{igusa2005}
Kiyoshi Igusa and Gordana Todorov.
\newblock On the finitistic global dimension conjecture for artin algebras.
\newblock {\em Representations of algebras and related topics}, 45:201--204,
  2005.

\bibitem{kershaw2023}
Luke Kershaw and Jeremy Rickard.
\newblock A finite dimensional algebra with infinite delooping level.
\newblock {\em arXiv preprint arXiv:2305.09109}, 2023.

\bibitem{krause2022symmetry}
Henning Krause.
\newblock On the symmetry of the finitistic dimension.
\newblock {\em arXiv preprint arXiv:2211.05519}, 2022.

\bibitem{mantese2004wakamatsu}
Francesca Mantese and Idun Reiten.
\newblock Wakamatsu tilting modules.
\newblock {\em Journal of Algebra}, 278(2):532--552, 2004.

\bibitem{rickard2019unbounded}
Jeremy Rickard.
\newblock Unbounded derived categories and the finitistic dimension conjecture.
\newblock {\em Advances in Mathematics}, 354:106735, 2019.

\bibitem{ringel2020gorenstein}
Claus~Michael Ringel and Pu~Zhang.
\newblock Gorenstein-projective and semi-gorenstein-projective modules.
\newblock {\em Algebra \& Number Theory}, 14(1):1--36, 2020.

\bibitem{smalo1998}
Sverre Smal{\o}.
\newblock The supremum of the difference between the big and little finitistic
  dimensions is infinite.
\newblock {\em Proceedings of the American Mathematical Society},
  126(9):2619--2622, 1998.

\bibitem{yamagata1996frobenius}
Kunio Yamagata.
\newblock Frobenius algebras.
\newblock {\em Handbook of algebra}, 1:841--887, 1996.

\end{thebibliography}

\end{document}